\documentclass[11pt,reqno]{amsart}

\usepackage{amsmath, amsthm, amsopn, amssymb}
\usepackage[nobysame,msc-links]{amsrefs}
\usepackage{enumerate, tikz, etoolbox, intcalc, geometry, caption, subcaption, afterpage}
\usepackage[english]{babel}
\usepackage{algorithmic,algorithm}

\newtheorem{theorem}{Theorem}[section]
\newtheorem{lemma}[theorem]{Lemma}
\newtheorem{proposition}[theorem]{Proposition}

\newtheorem{claim}[theorem]{Claim}
\newtheorem{question}[theorem]{Question}

\theoremstyle{definition}
\newtheorem{remark}[theorem]{Remark}

\newcommand{\cC}{\mathcal{C}}

\newcommand{\cL}{\mathcal{L}}
\newcommand{\cF}{\mathcal{F}}

\newcommand{\cS}{\mathcal{S}}
\newcommand{\cT}{\mathcal{T}}
\newcommand{\cI}{\mathcal{I}}
\newcommand{\cY}{\mathcal{Y}}
\newcommand{\cZ}{\mathcal{Z}}
\newcommand{\N}{\mathbb{N}}

\newcommand{\Z}{\mathbb{Z}}

\newcommand{\E}{\mathbb{E}}

\renewcommand{\Pr}{\mathbb{P}}
\newcommand{\1}{\mathbf{1}}

\newcommand{\zero}{\mathbf{0}}

\DeclareMathOperator\dist{dist}

\newcommand{\iid}{i.i.d.}
\newcommand{\ffiid}{ffiid}
\newcommand{\fvffiid}{fv-ffiid}

\title{Finitary coding for the sub-critical Ising model with finite expected coding volume}
\date{\today}

\author{Yinon Spinka}
\address{University of British Columbia.
    Department of Mathematics.
 	Vancouver, BC V6T 1Z2, Canada.}
    \email{yinon@math.ubc.ca}
\thanks{Research supported by Israeli Science Foundation grant 861/15, the European Research Council starting grant 678520 (LocalOrder), and the Adams Fellowship Program of the Israel Academy of Sciences and Humanities}

\begin{document}

\begin{abstract}
	It has been shown by van den Berg and Steif~\cite{van1999existence} that the sub-critical Ising model on~$\Z^d$ is a finitary factor of a finite-valued \iid\ process. We strengthen this by showing that the factor map can be made to have finite expected coding volume (in fact, stretched-exponential tails), answering a question of van den Berg and Steif.
	The result holds at any temperature above the critical temperature. An analogous result holds for Markov random fields satisfying a high-noise assumption and for proper colorings with a large number of colors.
\end{abstract}

\maketitle

\section{Introduction and main results}\label{sec:introduction}

Let $(S,\cS)$ and $(T,\cT)$ be two measurable spaces, and let $X=(X_v)_{v \in \Z^d}$ and $Y=(Y_v)_{v \in \Z^d}$ be $(S,\cS)$-valued and $(T,\cT)$-valued stationary random fields (i.e., $\Z^d$-processes) for some $d \ge 1$.
A \emph{coding} from $Y$ to $X$ is a measurable function $\varphi \colon T^{\Z^d} \to S^{\Z^d}$, which is \emph{translation-equivariant}, i.e., commutes with every translation of $\Z^d$, and which satisfies that $\varphi(Y)$ and $X$ are identical in distribution. Such a coding is also called a \emph{factor map} or \emph{homomorphism} from $Y$ to $X$, and when such a coding exists, we say that $X$ is a \emph{factor} of $Y$.

The \emph{coding radius} of $\varphi$ at a point $y \in T^{\Z^d}$, denoted by $R(y)$, is the minimal integer $r \ge 0$ such that $\varphi(y')_\zero=\varphi(y)_\zero$ for almost all $y' \in T^{\Z^d}$ which coincide with $y$ on the ball of radius $r$ around the origin in the graph-distance, i.e., $y'_v=y_v$ for all $v \in \Z^d$ such that $\|v\|_1 \le r$. It may happen that no such~$r$ exists, in which case, $R(y)=\infty$. Thus, associated to a coding is a random variable $R=R(Y)$ which describes the coding radius. We refer to $R^d$ as the \emph{coding volume}. A coding is called \emph{finitary} if $R$ is almost surely finite. When there exists a finitary coding from $Y$ to $X$, we say that $X$ is a \emph{finitary factor} of $Y$.

We say that a non-negative random variable $R$ has \emph{exponential tails} if $\Pr(R \ge r) \le Ce^{-cr}$ for some $C,c>0$ and all $r \ge 0$, and that it has \emph{stretched-exponential tails} if $\Pr(R \ge r) \le Ce^{-r^{c}}$ holds instead. When there exists a coding from $Y$ to $X$ whose coding radius has (stretched-)exponential tails, we say that $X$ is a finitary factor of $Y$ with (stretched-)exponential tails.

In this paper, we shall be concerned with finitary factors of \iid\ (independent and identically distributed) processes, distinguishing between the cases when the \iid\ process is finite-valued or infinite-valued. We use the abbreviation \emph{\ffiid} to denote a finitary factor of an \iid\ process (perhaps infinite-valued), and \emph{\fvffiid} to denote a finitary factor of a finite-valued \iid\ process.

Our main example is the (ferromagnetic) Ising model in $d \ge 2$ dimensions -- a classical discrete spin system in statistical mechanics. A \emph{Gibbs measure for the Ising model} on $\Z^d$ at inverse temperature $\beta>0$ is a probability measure $\mu$ on $\{-1,+1\}^{\Z^d}$ which satisfies that, if the random field $X=(X_v)_{v \in \Z^d}$ has distribution $\mu$, then for any vertex $v \in \Z^d$,
\begin{equation}\label{eq:ising-def}
\Pr\Big(X_v=\pm 1 ~\big|~ X|_{\Z^d \setminus \{v\}}\Big) = \frac{\exp\Big[\pm \beta \sum_{u \in N(v)} X_u \Big]}{\exp\Big[\beta \sum_{u \in N(v)} X_u \Big] + \exp\Big[-\beta \sum_{u \in N(v)} X_u \Big]} \qquad\text{almost surely} ,
\end{equation}
where $N(v)$ denotes the neighborhood of $v$. We note that although one usually defines Gibbs measures for the Ising model through their conditional distributions on any finite set, the above single-site specifications are sufficient as they determine the conditional finite-dimensional distributions.

It is well known~(see, e.g., \cite[Theorem~3.1]{georgii2001random} or \cite[pages 189-190 and 204]{liggett2012interacting}) that there exists a critical value $\beta_c(d) \in (0,\infty)$ such that there is a unique Gibbs measure for the Ising model on $\Z^d$ at inverse temperature $\beta<\beta_c(d)$ and multiple such Gibbs measures at inverse temperature $\beta>\beta_c(d)$.
Van den Berg and Steif~\cite{van1999existence} showed that the unique Gibbs measure in the former case is \fvffiid. We improve upon this and answer a question from~\cite{van1999existence} by showing the following.

\begin{theorem}\label{thm:ising}
Let $d \ge 2$ and let $\mu$ be the unique Gibbs measure for the Ising model on $\Z^d$ at inverse temperature $\beta<\beta_c(d)$. Then $\mu$ is \fvffiid\ with stretched-exponential tails.
\end{theorem}

\begin{remark}
	It has been shown in~\cite{van1999existence} that a phase transition (i.e., existence of multiple Gibbs measures) presents an obstruction for the existence of a finitary coding from an \iid\ process. In particular, at inverse temperature $\beta>\beta_c(d)$, no translation-invariant Gibbs measure for the Ising model is \ffiid.
\end{remark}

\begin{remark}\label{rem:ising-ffiid-exp-tails}
In the course of proving that $\mu$ is \fvffiid, it is shown in~\cite{van1999existence} that if one does not insist on a coding from a \emph{finite-valued} \iid\ process, then one may obtain a coding with exponential tails (so that $\mu$ is \ffiid\ with exponential tails); see Section~\ref{sec:pca} for more details. Similarly, it also follows that the critical Ising measure (i.e., when $\beta=\beta_c(d)$) is \ffiid\ (this relies on the fact that the phase transition is continuous~\cite{yang1952spontaneous,aizenman1986critical,aizenman2015random}), though it is shown in~\cite{van1999existence} (a result which was obtained jointly with Peres) that the coding volume $R^d$ cannot have finite expectation.
\end{remark}

\begin{remark}
	It is still unknown whether the critical Ising measure is \fvffiid.
\end{remark}

\begin{remark}
	In dimension $d=1$, the Ising model has a unique Gibbs measure $\mu$ at any finite inverse temperature $\beta$, and this measure is the distribution of an ergodic stationary Markov chain. It follows from a result in~\cite{harvey2006universal} that $\mu$ is \fvffiid\ with exponential tails (in fact, $\mu$ is a finitary factor with exponential tails of any \iid\ process with entropy strictly larger than that of $\mu$, and it is finitarily isomorphic to any \iid\ process with equal entropy~\cite{keane1979finitary}).
\end{remark}

\begin{remark}
	The FK random-cluster model is a dependent percolation model with infinite-range interactions, which is closely related to the Ising model. For background on this model, see~\cite{grimmett2006random}. Using the Edwards--Sokal coupling~\cite{swendsen1987nonuniversal,edwards1988generalization}, it is an easy consequence of Theorem~\ref{thm:ising} that the sub-critical FK-Ising measure (i.e., the random-cluster measure with parameters $q=2$ and any $p<p_c(q)$) is \fvffiid\ with stretched-exponential tails. Indeed, under this coupling, given the Ising configuration, the state of the edges in the random-cluster configuration are independent. It is shown in~\cite{harel2018finitary} that this result extends to the random-cluster model with any $q \ge 1$ and $p<p_c(q)$.
\end{remark}

Our second result concerns another well-known model of statistical mechanics -- proper colorings. Let $q \ge 3$. A \emph{proper $q$-coloring} of $\Z^d$ is a configuration $x \in \{1,\dots,q\}^{\Z^d}$ satisfying that $x_u \neq x_v$ for any adjacent vertices $u$ and $v$. A \emph{Gibbs measure for proper $q$-colorings} of $\Z^d$ is a probability measure $\mu$ on $\{1,\dots,q\}^{\Z^d}$ which is supported on proper $q$-colorings and satisfies that, if the random field $X=(X_v)_{v \in \Z^d}$ has distribution $\mu$, then for any finite set $\Lambda \subset \Z^d$, the conditional distribution of $X$ given its restriction to $\Lambda^c$ is uniform on the set of proper $q$-colorings which agree with $X$ on $\Lambda^c$.
It is well known (e.g., by Dobrushin uniqueness~\cite{salas1997absence}) that there is a unique Gibbs measure for proper $q$-colorings when $q>4d$. We show that this measure is \fvffiid\ with stretched-exponential tails when the number of colors is large enough.

\begin{theorem}\label{thm:proper-colorings}
Let $d \ge 2$ and $q \ge 4d(d+1)$. Let $\mu$ be the unique Gibbs measure for proper $q$-colorings of $\Z^d$. Then $\mu$ is \fvffiid\ with stretched-exponential tails.
\end{theorem}

\begin{remark}
The model of uniform proper $q$-colorings is equivalent to the zero-temperature anti-ferromagnetic $q$-state Potts model. It is intuitively clear that increasing the temperature only makes interactions weaker (the high-temperature model even satisfies high-noise; see below), and indeed, we believe that Theorem~\ref{thm:proper-colorings} extends to the anti-ferromagnetic $q$-state Potts model at any temperature (and $q$ as in the theorem above), though we do not pursue this here.
\end{remark}

Our third result is not about a particular model, but rather about a class of translation-invariant high-noise Markov random fields, which we proceed to define. Let $S$ be finite, let $\mu$ be a probability measure on $S^{\Z^d}$ and let $X=(X_v)_{v \in \Z^d}$ be distributed according to $\mu$. 
We say that $\mu$ is a \emph{Markov random field} if its conditional finite-dimensional distributions depend only on the immediate neighborhood of the finite set being inspected, i.e., if for any finite $V \subset \Z^d$ and any $\xi \in S^V$,
\[ \Pr\Big(X|_V=\xi ~\big|~ X|_{\Z^d \setminus V}\Big) = \Pr\Big(X|_V=\xi ~\big|~ X|_{\partial V}\Big) \qquad\text{almost surely} ,\]
where $\partial V$ denotes the set of vertices at distance 1 from $V$.
The Ising model and proper colorings (or rather the Gibbs measures for those models) are two examples of Markov random fields.

Suppose that $\mu$ is a translation-invariant Markov random field and, for $s \in S$, denote
\[ \gamma_s := \inf_{\substack{\xi \in S^{N(\zero)}\\\Pr(X|_{N(\zero)}=\xi)>0}} \Pr\Big(X_\zero = s ~\big|~ X|_{N(\zero)}=\xi\Big) .\]
We say that $\mu$ satisfies \emph{high-noise} if
\[ \gamma := \sum_{s \in S} \gamma_s > 1 - \frac{1}{2d} .\]
The quantity $\gamma$ is called the multigamma admissibility. It is essentially the probability that an update can be made to the spin at the origin without knowing anything about the values of the spins at its neighbors (see~\cite{haggstrom2000propp} for a more detailed explanation).
We remark that Dobrushin's uniqueness condition~\cite{Dobrushin1968TheDe} (or, alternatively, the ``disagreement percolation'' condition of van den Berg and Maes~\cite{van1994disagreement}) implies that if $\mu$ satisfies high-noise, then it is the only random field with the same conditional finite-dimensional distributions as~$\mu$.

\begin{theorem}\label{thm:high-noise-finitary-coding}
Let $\mu$ be a translation-invariant Markov random field satisfying high-noise. Then $\mu$ is \fvffiid\ with stretched-exponential tails.
\end{theorem}

Theorem~\ref{thm:high-noise-finitary-coding} improves on a result of H{\"a}ggstr{\"o}m and Steif~\cite{haggstrom2000propp} who showed that any translation-invariant high-noise Markov random field is \fvffiid.
Theorem~\ref{thm:high-noise-finitary-coding} applies to numerous models of statistical physics, including the Potts model (both ferromagnetic and anti-ferromagnetic) at high temperature, the hard-core model at low fugacity and the Widom--Rowlinson model at low fugacity (see~\cite{haggstrom2000propp} for more details on this for the Potts and Widom--Rowlinson models).
On the other hand, Theorem~\ref{thm:ising} does not follow from Theorem~\ref{thm:high-noise-finitary-coding}, as the Ising model does not satisfy high-noise when $\beta$ is only slightly smaller than $\beta_c(d)$.
Similarly, Theorem~\ref{thm:proper-colorings} does not follow from Theorem~\ref{thm:high-noise-finitary-coding} (even for large values of $q$), as it is clear that $\gamma=0$ for proper $q$-colorings, regardless of how large $q$ is.

The three theorems will be proved using a general result introduced in Section~\ref{sec:pca} about finitary codings for limiting distributions of probabilistic cellular automata.

\smallskip
\noindent
{\bf Background.}
We give here only a brief background and refer the reader to~\cite{van1999existence} for a more complete description of known results. A fundamental problem in ergodic theory is to understand which processes are isomorphic to which other processes (meaning that there is an almost everywhere invertible factor from one to the other). The very simplest of processes are the \iid\ processes, and therefore, of particular interest are those processes which are isomorphic to an \iid\ process; such processes are termed Bernoulli. The celebrated isomorphism theorem of Ornstein~\cite{ornstein1974} states that any two \iid\ processes of equal entropy are isomorphic (this result was later extended by Keane and Smorodinsky~\cite{keane1979bernoulli} who showed that any two such finite-valued processes are in fact finitarily isomorphic). Ornstein~\cite{ornstein1974} further showed that any factor of an \iid\ process is Bernoulli. This shed a more probabilistic light on the notion of Bernoullicity.

The notion of a finitary factor of \iid\ has the advantage that it allows to compute a symbol in the target process by only revealing (almost surely) finitely many variables of the \iid\ process. This gives a more concrete construction of the target process, which may also be useful for exact simulation algorithms.
Besides this appealing feature, finitary factors of \iid\ have particular relevance in the context of probabilistic models such as those considered here. Let us take the Ising model as an example. It has been shown~\cite{OW73} (see also~\cite{adams1992folner}) that the so-called ``plus state'' (this is the Gibbs measure obtained by taking $+$ boundary conditions) is a factor of an \iid\ process (i.e., is Bernoulli) for any value of the inverse temperature~$\beta$. Thus, the phase transition is not reflected in this notion. However, as shown in~\cite{van1999existence}, it is indeed reflected in the notion of a \emph{finitary} factor: the ``plus state'' is a finitary factor of an \iid\ process when $\beta<\beta_c(d)$, but not when $\beta>\beta_c(d)$.

In constructing a finitary coding from an \iid\ process to a given process, it is desirable for efficiency purposes (e.g., for simulation algorithms) that the \iid\ process be ``small'' and that the coding radius also be typically small. One such qualitative meaning of this is that the \iid\ process is finite-valued and that the coding volume has finite expectation. A more quantitative meaning of this would be to require bounds on the entropy of the \iid\ process and on the tail of the coding radius. Our results are a mixture of the two as they yield a finitary coding from a finite-valued \iid\ process with stretched-exponential tails for the coding radius. In particular, our result about the Ising model (Theorem~\ref{thm:ising}) answers a question of van den Berg and Steif~\cite[Question~2]{van1999existence}, who asked whether the sub-critical Ising measure is \fvffiid\ with finite expected coding volume.

\smallskip
\noindent
{\bf Notation.}
We consider $\Z^d$ as a graph in which two vertices $u$ and $v$ are adjacent if $|u-v|=1$, where $|v|=\|v\|_1:=|v_1|+\cdots+|v_d|$ denotes the $\ell_1$-norm. We denote by $N(v) := \{ u \in \Z^d : |u-v|=1 \}$ the neighborhood of $v$. For a set $U \subset \Z^d$, we write $\dist(v,U) := \min_{u \in U} |u-v|$. We use $\zero$ to denote the origin $(0,\dots,0) \in \Z^d$ and $e_1:=(1,0,\dots,0) \in \Z^d$. We use $\N$ to denote the non-negative integers.

\smallskip
\noindent
{\bf Organization.}
In Section~\ref{sec:pca}, we formulate the result about finitary codings for limiting distributions of probabilistic cellular automata (Theorem~\ref{thm:main}) and use it to prove Theorem~\ref{thm:ising}, Theorem~\ref{thm:proper-colorings} and Theorem~\ref{thm:high-noise-finitary-coding}. In Section~\ref{sec:stopping-processes}, we introduce an abstract tool (Proposition~\ref{prop:coding-simple} and the more general Proposition~\ref{prop:coding}) and show how to deduce Theorem~\ref{thm:main} from it. In Section~\ref{sec:algo}, we introduce and explain an algorithm, which is then used in Section~\ref{sec:proof} to prove Proposition~\ref{prop:coding}. We end with open questions in Section~\ref{sec:open}.

\smallskip
\noindent
{\bf Acknowledgments.}
I would like to thank Nishant Chandgotia, Peleg Michaeli, Ron Peled and Jeff Steif for useful discussions and comments, and Matan Harel for help in proving Lemma~\ref{lem:stretched-exp-decay}. I am also grateful to the anonymous referee for suggestions which greatly improved the presentation.

\section{Finitary codings for limiting distributions of PCAs}
\label{sec:pca}

The goal of this section is to define the notion of a \emph{probabilistic cellular automaton (PCA)} and other relevant notions, formulate a general result about finitary codings for limiting distributions of PCAs (Theorem~\ref{thm:main} below), and then use this theorem to deduce the results stated in Section~\ref{sec:introduction}.

Before doing so, we give an informal description of the relevant ideas and concepts in the case of the Ising model: Consider the continuous-time Glauber dynamics for the sub-critical Ising model -- each vertex has an exponential clock (with rate 1), and when its clock rings, it updates its spin value according to the conditional distribution given by the values of its neighbors as in~\eqref{eq:ising-def}. This is an ergodic process, whose unique stationary measure is $\mu$ (of Theorem~\ref{thm:ising}), and thus, the distribution at time $t$ converges to $\mu$ as $t \to \infty$, regardless of the initial configuration. As we are interested in finding a coding from a finite-valued process, we instead opt to use a discrete analogue of these dynamics, given by a PCA: at each discrete time step $n$, every vertex is independently set to active or inactive with some fixed probability, and every active vertex which has no active neighbors then updates its spin value as before.
This too is an ergodic process and the distribution at time $n$ converges to $\mu$ as $n \to \infty$. Convergence alone is not sufficient to obtain a coding of $\mu$, as the latter requires an exact sample from $\mu$. To get such a sample, one can employ the coupling-from-the-past technique of Propp and Wilson~\cite{propp1996exact}. This then yields a finitary coding for $\mu$ from an infinite-valued \iid\ process (showing that $\mu$ is \ffiid). Using a result of Martinelli and Olivieri~\cite{martinelli1994approach} that the convergence of the above process to stationarity occurs at an exponential rate, one may further show that this coding has a coding radius with exponential tails (showing that $\mu$ is \ffiid\ with exponential tails). To get from this a (finitary) coding from a \emph{finite-valued} \iid\ process, still requires quite some work. All the above, including this last step, has been carried out by van den Berg and Steif~\cite{van1999existence}. Thus, they showed that $\mu$ is \fvffiid. However, they gave no information on the coding radius beyond its almost sure finiteness. Our main contribution is to show how one can carry out this last step in a controlled manner which preserves the good tails of the coding radius (yielding stretched-exponential tails). We elaborate on this in the next sections.

The above includes general arguments about certain dynamics, along with some model-specific information. Indeed, van den Berg and Steif separated the two parts of the argument, and proved the more general result~\cite[Theorem~3.4]{van1999existence} that the limiting distribution of a monotone, exponentially ergodic PCA is \fvffiid.
In order to accommodate for the different situations considered in Section~\ref{sec:introduction}, which include non-monotone models (proper colorings and high-noise Markov random fields), we work here in the more general setting of exponentially uniformly ergodic PCAs (instead of monotone, exponentially ergodic PCAs), defined below. The proof of~\cite[Theorem~3.4]{van1999existence} may be extended to this setting to show that the limiting distribution of an exponentially uniformly ergodic PCA is \fvffiid.
As mentioned before, the main challenge, and our primary contribution, is to show that this can be done while simultaneously controlling the coding radius.

\begin{theorem}\label{thm:main}
The limiting distribution of an exponentially uniformly ergodic PCA is \fvffiid\ with stretched-exponential tails.
\end{theorem}

The results of Section~\ref{sec:introduction} will follow from Theorem~\ref{thm:main} by showing that the corresponding measures are limiting distributions of exponentially uniformly ergodic PCAs. Theorem~\ref{thm:main} will be proved in Section~\ref{sec:stopping-processes}.
Let us also mention the following result which will easily follow from our definition of an exponentially uniformly ergodic PCA (unlike Theorem~\ref{thm:main} which requires work).

\begin{theorem}\label{thm:main0}
The limiting distribution of an exponentially uniformly ergodic PCA is \ffiid\ with exponential tails.
\end{theorem}

We emphasize the differences between the two theorems: the second gives a finitary coding with exponential tails but does not provide any control on the \iid\ process, while the first gives a coding from a finite-valued \iid\ process but does slightly worse in terms of the tails of the coding radius.

\medskip

Let us now proceed to give precise definitions. We begin by defining what a \emph{PCA} is.
For our purposes, a PCA is a discrete-time evolution on $S^{\Z^d}$ for some non-empty finite set $S$, which can be described as follows.
Let $(W_{v,i})_{v \in \Z^d, i \in \Z}$ be a collection of \iid\ random variables taking values in a finite set~$A$. Let $F,F' \subset \Z^d$ be finite and let $f \colon S^F \times A^{F'} \to S$. The time evolution started from $\xi \in S^{\Z^d}$ is the process $(\omega_{v,i})_{v \in \Z^d,i \ge 0}$ defined by
\begin{equation}\label{eq:time-evolution}
\begin{aligned}
	\omega_{v,0} &:= \xi_v, &&v \in \Z^d ,\\
	\omega_{v,i+1} &:= f\big( (\omega_{v+u,i})_{u \in F}, (W_{v+u,i})_{u \in F'} \big) , &&v \in \Z^d,~i \ge 0.
\end{aligned}
\end{equation}
We stress that different choices of $W_{v,i}$ and $f$ could give rise to the same time evolutions (i.e., the same distribution), however, for our purposes, a PCA is the data of the distribution of the $W_{v,i}$, the sets $F$ and $F'$ and the function $f$. In particular, we note that a PCA comes equipped with a simultaneous coupling of the time evolutions started from all starting states $\xi$. We remark that the usual definition of a PCA requires that $F'=\{\zero\}$, in which case, conditioned on $\{\omega_{v,i}\}_v$, the random variables $\{\omega_{v,i+1}\}_v$ are mutually independent. For the above approach to the construction of finitary codings, we will have $F=F'=N(\zero) \cup \{\zero\}$ (recall that $N(\zero)$ is the neighborhood of the origin), in which case, there are local conditional dependencies. 

A PCA is said to be \emph{ergodic} if there exists a probability measure $\mu$ on $S^{\Z^d}$ such that, for any starting state $\xi$, the distribution of $(\omega_{v,i})_{v \in \Z^d}$ converges weakly to $\mu$ as $i \to \infty$.
An ergodic PCA converging to $\mu$ can be used to obtain an approximate sample from $\mu|_\Lambda$, the marginal of $\mu$ on a finite subset $\Lambda$ of $\Z^d$, by running the time evolution of the PCA until some large time $t$ and observing the restricted process $(\omega_{v,t})_{v \in \Lambda}$ at that time, noting also that the latter is determined by a finite collection of random variables, namely,
\begin{equation}\label{eq:time-evolution-determined}
(\omega_{v,t})_{v \in \Lambda}\text{ is determined by $\xi$ and }\{W_{v,t-i}\}_{\dist(v,\Lambda) \le \Delta i, 1 \le i \le t} , \qquad\text{where }\Delta := \max_{u \in F \cup F'} \|u\|_1 .
\end{equation}
As is usual in these situations, determining how large $t$ should be in order to obtain a sample whose distribution is close to the limiting distribution, is not an easy task.

One way around this is to devise a method to exactly sample from the limiting distribution. Coupling-from-the-past provides such a method, at the cost, however, of requiring a type of uniform ergodicity. To define this notion, we first extend the definition given in~\eqref{eq:time-evolution} of the time evolution of the PCA to allow starting at any integer time as follows.
The time evolution started from $\xi \in S^{\Z^d}$ at time $i_0 \in \Z$ is the process $(\omega_{v,i}^{\xi,i_0})_{v \in \Z^d, i \ge i_0}$ defined by
\begin{equation}\label{eq:time-evolution-past}
\begin{aligned}
	\omega_{v,i_0}^{\xi,i_0} &:= \xi_v, &&\quad v \in \Z^d ,\\
	\omega_{v,i+1}^{\xi,i_0} &:= f\big( (\omega_{v+u,i}^{\xi,i_0})_{u \in F}, (W_{v+u,i})_{u \in F'} \big) , &&\quad v \in \Z^d,~i \ge i_0.
\end{aligned}
\end{equation}
We say that an ergodic PCA is \emph{uniformly ergodic} if
\begin{equation}\label{eq:coupling-from-the-past-stopping-time}
\tau_v := \min \big\{ i \ge 0 : \omega^{\xi,-i}_{v,0}\text{ does not depend on the starting state }\xi \big\}
\end{equation}
is almost surely finite for all $v$.
We remind the reader that in our definitions, a PCA always comes equipped with a function $f$, so that the notion of uniform ergodicity depends on this $f$. While this might not be the standard notion of uniform ergodicity, it will be the relevant one for us. We also remark that for monotone PCAs, ergodicity implies uniform ergodicity (see~\cite[Lemma~3.5]{van1999existence}).
We say that an ergodic PCA is \emph{exponentially uniformly ergodic} if $\tau_v$ has exponential tails.
For a uniformly ergodic PCA, we define the random field $\omega^* = (\omega^*_v)_{v \in \Z^d}$ by
\begin{equation}\label{eq:coupling-from-the-past-process}
\omega^*_v := \omega^{\xi,-\tau_v}_{v,0} ,\qquad v \in \Z^d ,
\end{equation}
noting that this is almost surely well-defined and does not depend on $\xi$. We point out that while earlier we needed $W_{v,i}$ with $i \ge 0$, for~\eqref{eq:coupling-from-the-past-stopping-time} and~\eqref{eq:coupling-from-the-past-process} we use $W_{v,i}$ with $i<0$.

The following proposition encompasses the essence of coupling-from-the-past (in its infinite-volume version). An analogous statement for monotone ergodic PCAs was shown in~\cite{van1999existence} and a similar statement for PCAs arising from high-noise Markov random fields was shown in~\cite{haggstrom2000propp}. The proofs of these statements are easily adapted to the setting described here, and we include a short proof for completeness.

\begin{proposition}\label{prop:coupling-from-the-past}
	Suppose $\mu$ is the limiting distribution of a uniformly ergodic PCA having time evolution $\omega$. Then $\omega^*$ has distribution $\mu$.
\end{proposition}
\begin{proof}
	Fix a finite $\Lambda \subset \Z^d$ and denote $\Omega^t := (\omega^{\xi,-t}_{v,0})_{v \in \Lambda}$.
	Note that, since $\omega^{\xi,-\tau_v}_{v,0}=\omega^*_v$ for all $\xi$, it follows from~\eqref{eq:time-evolution-past} that $\omega^{\xi,-t}_{v,0}=\omega^*_v$ for all $t \ge \tau_v$.
	In particular, $\Omega^t = \omega^*|_\Lambda$ for all $t \ge \max_{v \in \Lambda} \tau_v$. Since $(\omega^{\xi,-t}_{v,0})_{v \in \Z^d}$ and $(\omega^{\xi,0}_{v,t})_{v \in \Z^d}$ are identical in distribution for any $t \ge 0$, it follows that $\Omega^t$ converges to $\mu|_\Lambda$ in distribution as $t \to \infty$. On the other hand, as we have seen that $\Omega^t$ eventually equals $\omega^*|_\Lambda$, we conclude that $\omega^*|_\Lambda$ has distribution $\mu|_\Lambda$. Since $\Lambda$ was arbitrary, the proposition follows.
\end{proof}

The proposition implies that the limiting distribution of a uniformly ergodic PCA is \ffiid. Indeed, a moment of thought reveals that~\eqref{eq:time-evolution-past}-\eqref{eq:coupling-from-the-past-process} describe such a finitary coding from the process $((W_{v,i})_{i<0})_{v \in \Z^d}$. Moreover, if the PCA is exponentially uniformly ergodic, then the coding radius of this coding has exponential tails, so that the limiting distribution is in fact \ffiid\ with exponential tails. This establishes Theorem~\ref{thm:main0}. The point is, however, that this coding is not from a finite-valued process. Restricting the \iid\ process to be finite-valued, while keeping control of the coding radius, is the missing step in order to establish Theorem~\ref{thm:main} and is what most of the remainder of the paper is devoted to.
Before coming back to this in the next section, we explain how to deduce the results of Section~\ref{sec:introduction} from Theorem~\ref{thm:main}.

We now prove Theorem~\ref{thm:ising}, Theorem~\ref{thm:proper-colorings} and Theorem~\ref{thm:high-noise-finitary-coding}. In light of Theorem~\ref{thm:main}, this boils down to showing that in each case the corresponding measure is the limiting distribution of an exponentially uniformly ergodic PCA.

\subsection{The Ising model -- proof of Theorem~\ref{thm:ising}}
To deduce from Theorem~\ref{thm:main} that the sub-critical Ising measure is \fvffiid\ with stretched-exponential tails, we must know that it is the limiting distribution of an exponentially uniformly ergodic PCA. This was shown by van den Berg and Steif (see the proof of Theorem~4.1 in~\cite{van1999existence}) who relied on a deep result of Martinelli and Olivieri~\cite{martinelli1994approach} about the continuous-time Glauber dynamics for the Ising model (see also \cite[Proposition~4.2]{van1999existence}).

\begin{proposition}[\cite{van1999existence}]\label{prop:ising-exponentially-ergodic}
	Let $d \ge 2$ and let $\mu$ be the unique Gibbs measure for the Ising model on $\Z^d$ at inverse temperature $\beta<\beta_c(d)$. Then $\mu$ is the limiting distribution of an exponentially uniformly ergodic PCA.
\end{proposition}

Theorem~\ref{thm:ising} follows immediately from Theorem~\ref{thm:main} and Proposition~\ref{prop:ising-exponentially-ergodic}.

Although for our purposes, we only need to know that a PCA as in Proposition~\ref{prop:ising-exponentially-ergodic} exists and its details are not important for us, in order to provide the reader with a full picture for the case of the Ising model, we nevertheless give a complete and formal description of the PCA used in the proof of Proposition~\ref{prop:ising-exponentially-ergodic} (but written in a slightly different way than in~\cite{van1999existence}). In fact, we have already given an informal description of this PCA in the beginning of Section~\ref{sec:pca}.
To define it precisely, set $S:=\{-1,1\}$, $A:=\{0,1\} \times \{-2d,\dots,2d+1\}$, $F=F':= N(\zero) \cup \{\zero\}$ and define $f \colon S^F \times A^F \to S$ by
\[ f(\eta, (\phi,\psi)) := \begin{cases}
2 \cdot \1_{\big\{\psi_\zero \le \sum_{u \in N(\zero)} \eta_u \big\}} - 1 &\text{if }\phi_\zero=1\text{ and }\phi_v=0\text{ for all }v \in N(\zero)\\
\eta_0 &\text{otherwise}
\end{cases}.\]
To complete the description of the PCA, we must also describe the distribution of the \iid\ random variables $(W_{v,i})_{v \in \Z^d, i \in \Z}$. We let each $W_{v,i}$ consist of a pair of independent random variables, the first of which is a Bernoulli random variable with parameter, say, $1/2$, and the second of which has the distribution of $\mathcal{W}$, where $\mathcal{W}$ takes values in $\{-2d,\dots,2d+1\}$ and satisfies
\[ \Pr(\mathcal{W} \le k) = p_k := \frac{e^{\beta k}}{e^{\beta k} + e^{-\beta k}} \qquad\text{for }-2d \le k \le 2d .\]
Observe that such a random variable exists since $(p_k)_{-2d \le k \le 2d}$ is increasing.
Recalling~\eqref{eq:ising-def}, one may easily verify that any Gibbs measure for the Ising model on $\Z^d$ at inverse temperature $\beta$ is a stationary measure for this PCA.

We note that this PCA is monotonic in the sense that $f(\eta, (\phi,\psi)) \le f(\eta', (\phi,\psi))$ for any $(\phi,\psi)$ and $(\eta,\eta')$ such that $\eta_v \le \eta'_v$ for all $v \in F$, and we remark that due to this monotonicity, Proposition~\ref{prop:ising-exponentially-ergodic} is essentially a statement about the probability that the value of the spin at the origin after time $t$ depends on whether the starting state is the constant plus or constant minus state -- namely, that this probability is exponentially small in $t$.

\subsection{High-noise Markov random fields -- proof of Theorem~\ref{thm:high-noise-finitary-coding}}

To deduce Theorem~\ref{thm:high-noise-finitary-coding} from Theorem~\ref{thm:main}, we need to know that a translation-invariant high-noise Markov random field is the limiting distribution of an exponentially uniformly ergodic PCA. This was shown by H{\"a}ggstr{\"o}m and Steif in~\cite{haggstrom2000propp} (essentially Proposition~2.1 there).

\begin{proposition}[\cite{haggstrom2000propp}]\label{thm:high-noise-PCA}
Let $\mu$ be a translation-invariant Markov random field satisfying high-noise. Then $\mu$ is the limiting distribution of an exponentially uniformly ergodic PCA.
\end{proposition}

Given this proposition, Theorem~\ref{thm:high-noise-finitary-coding} is an immediate corollary of Theorem~\ref{thm:main}.

\subsection{Proper colorings -- proof of Theorem~\ref{thm:proper-colorings}}
Theorem~\ref{thm:proper-colorings} will follow from Theorem~\ref{thm:main} once we establish the following.

\begin{proposition}\label{thm:proper-colorings-PCA}
Let $d \ge 2$ and $q \ge 4d(d+1)$. Let $\mu$ be the unique Gibbs measure for proper $q$-colorings of $\Z^d$. Then $\mu$ is the limiting distribution of an exponentially uniformly ergodic PCA.
\end{proposition}

\begin{proof}
The proof uses ideas of Huber~\cite{huber1998exact,huber2004perfect} for exact sampling of proper colorings on finite graphs and ideas of H{\"a}ggstr{\"o}m and Steif~\cite{haggstrom2000propp} from the proof of Proposition~\ref{thm:high-noise-PCA}.
The proof of the latter proposition uses an auxiliary PCA on a larger space (which the authors there call a super-PCA), which ``bounds'' the original PCA simultaneously for all starting states, and thus allows to ``detect'' when the original PCA has coalesced.
Huber used a similar idea (which he called bounding chains), together with model-specific arguments, to provide an exact sampling algorithm for proper colorings (and other models) on a finite graph.
Putting these ideas together, we show how this can be done for proper colorings of $\Z^d$.

We first describe the PCA in words: at each time step, every vertex is independently set to active or inactive with some fixed probability, and every active vertex which has no active neighbors then updates its color to be uniformly chosen from the set of colors not appearing at any of its neighbors. More precisely, a uniform permutation of the colors is chosen, and the first color not appearing at any neighbor is chosen.
This PCA may be realized as follows. Let $S := \{1,\dots,q\}$ and let $\cS_q$ be the symmetric group on $S$. Let $F=F':= N(\zero) \cup \{\zero\}$, $A := \{0,1\} \times \cS_q$ and define $f \colon S^F \times A^F \to S$ by
\[ f(\eta, (\phi,\psi)) := \begin{cases}
 g(\{\eta_v\}_{v \in N(\zero)}, \psi_\zero) &\text{if }\phi_\zero=1\text{ and }\phi_v=0\text{ for all }v \in N(\zero)\\
 \eta_\zero &\text{otherwise}
 \end{cases},\]
where $g \colon 2^S \times \cS_q \to S$ is defined by
\[ g(D,\pi) := \pi(\min\{ i \in S : \pi(i) \notin D \}) .\]
The time evolution $\omega$ of this PCA is then given by~\eqref{eq:time-evolution-past}, where the \iid\ random variables $(W_{v,i})$ are chosen to be uniformly distributed over $A$, so that each $W_{v,i}$ represents an unbiased coin toss (the unbiasedness will not be important for us) and an independent uniformly chosen permutation of the colors.
It is straightforward to check that any Gibbs measure for proper $q$-colorings is a stationary distribution for this PCA.

To show that this PCA is exponentially uniformly ergodic, we use the method of bounding chains discussed above. Consider the following PCA (or super-PCA in the language of~\cite{haggstrom2000propp}) on $(2^S)^{\Z^d}$ given by $\hat{f} \colon (2^S)^F \times A^F \to 2^S$, where
\[ \hat{f}(\hat{\eta}, (\phi,\psi)) := \begin{cases}
\hat{g}(\hat{\eta}|_{N(\zero)}, \psi_\zero) &\text{if }\phi_\zero=1\text{ and }\phi_v=0\text{ for all }v \in N(\zero)\\
\hat{\eta}_\zero &\text{otherwise}
\end{cases},\]
where $\hat{g} \colon (2^S)^{N(\zero)} \times \cS_q \to S$ is defined by
\[ \hat{g}(\hat{\eta},\pi) := \bigcup_{\substack{\eta \in S^{N(\zero)}:\\\eta_v \in \hat{\eta}_v~\forall v \in N(\zero)}} g(\{\eta_v\}_{v \in N(\zero)},\pi) .\]
The time evolution $\hat{\omega}$ of this PCA is then defined as in~\eqref{eq:time-evolution-past}, using the same random variables $(W_{v,i})$ as above, so that the two PCAs are coupled, with the crucial property that $\hat{\omega}$ bounds $\omega$ in the sense that
\[ \omega^{\xi,i}_{v,j} \in \hat{\omega}^{\hat{\xi},i}_{v,j} \qquad\text{for any $\xi \in S^{\Z^d}$, $v \in \Z^d$ and $i \le j$}, \]
where $\hat{\xi}$ is the maximal element in $(2^S)^{\Z^d}$ defined by $\hat{\xi}_v := 2^S$ for all $v \in \Z^d$.
In particular, recalling~\eqref{eq:coupling-from-the-past-stopping-time}, we have
\[ \tau_v \le \hat{\tau}_v = \min \big\{ i \ge 0 : |\hat{\omega}^{\hat{\xi},-i}_{v,0}|=1 \big\} .\]
It therefore suffices to show that $\hat{\tau}_v$ has exponential tails.

We begin by observing that, for $t \ge 0$,
\[ \Pr(\hat{\tau}_v > t) \le \Pr\Big(|\hat{\omega}^{\hat{\xi},-t}_{v,0}|>1\Big) = \Pr\Big(|\hat{\omega}^{\hat{\xi},0}_{v,t}|>1\Big) =: p_t ,\]
where $p_t$ is of course independent of $v$.
To ease notation, let us denote $Y_t(v) := \hat{\omega}^{\hat{\xi},0}_{v,t}$.
Let $\alpha$ denote the probability that a vertex is updated in any given time step, i.e., $\alpha=\beta(1-\beta)^{2d}$, where $\beta$ is the probability that a vertex is activated (for an unbiased coin toss, we have $\beta=1/2$ and $\alpha=2^{-2d-1}$, but this will not be used). Note that $|\hat{g}(\hat{\eta},\pi)| \le 2d+1$ for any $\hat{\eta} \in (2^S)^{N(\zero)}$ and $\pi \in \cS_q$. Hence,
\begin{equation}\label{eq:proper-colorings-at-most-2d+1}
\Pr(|Y_t(u)| > 2d+1) = (1-\alpha)^t \quad\text{for any $u \in \Z^d$ and $t \ge 0$,}
\end{equation}
since $|Y_t(u)| > 2d+1$ if and only if $u$ has never been updated by time $t$.

Let us see what happens when the origin is updated.
Let $D := \bigcup_{u \in N(\zero)} Y_t(u)$ be the set of colors which may appear in some neighbor of $\zero$, and let $D' := \bigcup_{u \in N(\zero), |Y_t(u)|=1} Y_t(u)$ be those colors which are known to appear in some neighbor of $\zero$. Observe that if $\pi \in \cS_q$ is such that $g(D',\pi) \notin D$, then $g(D,\pi)=g(D',\pi)$ and $\hat{g}((Y_t)|_{N(\zero)},\pi)=\{g(D',\pi)\}$. Thus, given $Y_t$ and given that $\zero$ is updated at time $t+1$, the probability that $|Y_{t+1}(\zero)|>1$ is at most the probability that the $g(D',\pi) \in D$. When $\pi \in \cS_q$ is chosen uniformly, $g(D',\pi)$ is uniformly distributed in $S \setminus D'$, so that the latter probability is $\frac{|D \setminus D'|}{q-|D'|} \le \frac{\sum_{u \in N(\zero)} |Y_t(u)|\1_{\{|Y_t(u)|>1\}}}{q-2d+1}$.
This shows that
\[ \Pr(|Y_{t+1}(\zero)|>1 \mid Y_t) \le (1-\alpha) \1_{\{|Y_t(\zero)|>1\}} + \alpha \sum_{u \in N(\zero)} \frac{|Y_t(u)|\1_{\{|Y_t(u)|>1\}}}{q-2d+1} .\]
Together with~\eqref{eq:proper-colorings-at-most-2d+1}, this yields
\[ p_{t+1} = \E \big[ \Pr(|Y_{t+1}(\zero)|>1 \mid Y_t) \big] \le 2d (1-\alpha)^t + \left(1 - \alpha\left(1-\frac{2d(2d+1)}{q-2d+1}\right) \right) p_t .\]
Thus, $p_t$ decays exponentially in $t$ when $q \ge 4d(d+1)$, and Proposition~\ref{thm:proper-colorings-PCA} follows.
\end{proof}

\section{A general result and proof of Theorem~\ref{thm:main}}
\label{sec:stopping-processes}

In this section, we introduce a general result which will allow us to deduce Theorem~\ref{thm:main}. This result is an abstract tool and is not, a priori, related to the problems originally discussed in Section~\ref{sec:introduction}.

Let $X=(X_{v,i})_{v \in \Z^d, i \ge 0}$ be a process taking values in a finite set $S$.
Let $B=(B_n)_{n \ge 0}$ be a strictly increasing sequence of subsets of $\Z^d \times \N$ with $B_0:=\{(\zero,0)\}$, and consider the associated $\sigma$-algebras $\{\cF^n_v\}_{v \in \Z^d, n \ge 0}$ defined by
\begin{equation}\label{eq:F_n}
	\cF^n_v := \sigma\big(\{X_{v+u,i}\}_{(u,i) \in B_n}\big) .
\end{equation}
An $\N$-valued random field $\tau=(\tau_v)_{v \in \Z^d}$ is said to be a \emph{$B$-stopping-process} for $X$ if, for every $v$, $\tau_v$ is an almost surely finite stopping time with respect to the filtration $(\cF^n_v)_{n \ge 0}$.
When we say that such a stopping-process is stationary, we shall mean that the same stopping rule is used at every vertex (rather than just meaning that its law is translation-invariant).
Given a $B$-stopping-process~$\tau$, we denote by $X^\tau$ the random field
\[ X^\tau := \big((X_{v+u,i})_{(u,i) \in B_{\tau_v}}\big)_{v \in \Z^d} .\]
Note that $(X^\tau)_v$ takes values in the finite-configuration space $\bigcup_{n \ge 0} S^{B_n}$.
We say that $B$ is \emph{linear} if
\begin{equation}\label{eq:linear-stopping-process}
\Delta_n := \max \big\{ \max\{|u|,i\} : (u,i) \in B_n\big\} \le \Delta n \qquad\text{for some $\Delta \ge 1$ and all $n \ge 0$.}
\end{equation}

\begin{proposition}\label{prop:coding-simple}
Let $X=(X_{v,i})_{v \in \Z^d, i \ge 0}$ be a finite-valued \iid\ process, let $B$ be linear and let~$\tau$ be a stationary $B$-stopping-process for $X$. Suppose $\tau_v$ has exponential tails and $\E |B_{\tau_v}| < M$ for some integer $M$. Then $X^\tau$ is a finitary factor of $((X_{v,i})_{0 \le i < M})_{v \in \Z^d}$ with stretched-exponential tails.
\end{proposition}

Before using Proposition~\ref{prop:coding-simple} to prove Theorem~\ref{thm:main}, we briefly explain the proposition and how it relates to the setting of the theorem.
Recall that, given a uniformly ergodic PCA, \eqref{eq:time-evolution-past}-\eqref{eq:coupling-from-the-past-process} explicitly express the random field $\omega^*$ as a finitary factor of the \iid\ process $((W_{v,i})_{i<0})_{v \in \Z^d}$, defined via certain stopping times. Moreover, it is clear from this and from~\eqref{eq:time-evolution-determined} that the value of the output $\omega^*_u$ for any given $u$ depends only on the variables $W_{v,i}$ within a certain ``cone'' in space-time emanating from $(u,0)$ (this is because as one goes back in time, the spatial dependency grows linearly). The above setup generalizes this situation to an abstract setting (which has nothing to do with coupling-from-the-past or PCAs), where the sequence $(B_n)_n$ replaces the cones arising from~\eqref{eq:time-evolution-determined}, the stopping process replaces the coupling-from-the-past stopping times given in~\eqref{eq:coupling-from-the-past-stopping-time}, and the variables $(W_{v,i})_{v \in \Z^d,i<0}$ are now called $(X_{v,i})_{v \in \Z^d,i \ge 0}$. With this interpretation in mind, for any given $u$, we may think of $(X^\tau)_u$ as containing all the variables that are ``needed'' for the computation of the output at $u$, and the proposition states (ignoring the tails of $\tau_v$ and the coding) that if, on average, the number of variables needed to compute the output at a given vertex is less than $M$, then one can ``emulate'' the process $X^\tau$ (consisting of all the needed variables) from a process which has precisely $M$ variables at each vertex. In other words, if one has an algorithm which can a priori need access to any number of variables at a given vertex, but typically does not need many such variables, then by ``transporting'' variables from one space-time location to another as needed, it is possible to rewrite the algorithm in such a way that it only has access to a bounded number of variables at each vertex.
We note that we continue to refer to $\Z^d \times \N$ as space-time, although the interpretation of $\N$ as a time dimension is perhaps less proper.

\begin{proof}[Proof of Theorem~\ref{thm:main}]
	Suppose that $\mu$ is the limiting distribution of an exponentially uniformly ergodic PCA with time evolution $\omega$, defined via variables $(W_{v,i})$, sets $F$ and $F'$, and function $f$. By Proposition~\ref{prop:coupling-from-the-past}, it suffices to show that $\omega^*$, defined by~\eqref{eq:coupling-from-the-past-process}, is \fvffiid\ with stretched-exponential tails.

Recall the definition of $\tau_v$ from~\eqref{eq:coupling-from-the-past-stopping-time} and the definition of $\Delta$ from~\eqref{eq:time-evolution-determined}.
By definition of $\tau_v$ and~\eqref{eq:time-evolution-determined} (or rather the analogue of~\eqref{eq:time-evolution-determined} for the time evolution started at time $-t$ and run up to time 0), the value of $\omega^*_v$ is a deterministic function of the variables $(W_{v+u,-i})_{|u| \le \Delta i, 0 \le i \le \tau_v}$ (actually, the variable $W_{v,0}$ corresponding to $i=0$ is not needed, but we include it nevertheless). Moreover, this function does not depend on $v$, in the sense that, for some deterministic function $\psi$, we have that $\omega^*_v = \psi((W_{v+u,-i})_{|u| \le \Delta i, 0 \le i \le \tau_v})$ for all $v$.

Towards applying Proposition~\ref{prop:coding-simple}, define $B=(B_n)_{n \ge 0}$ by $B_n := \{ (u,i) : |u| \le \Delta i,\, 0 \le i \le n \}$ and define the \iid\ process $X=(X_{v,i})_{v\in\Z^d,i \ge 0}$ by $X_{v,i} := W_{v,-i}$. Note that $\tau=(\tau_v)_{v \in \Z^d}$ is a linear stationary $B$-stopping-process for $X$, and that $\omega^*_v = \psi((X^\tau)_v)$ so that $\omega^*$ is a finitary factor of $X^\tau$ with coding radius 0. It therefore suffices to show that $X^\tau$ is \fvffiid\ with stretched-exponential tails.
Indeed, letting $M$ be any integer larger than $\E |B_{\tau_v}|$, Proposition~\ref{prop:coding-simple} yields that $X^\tau$ is a finitary factor of $((X_{v,i})_{0 \le i < M})_{v \in \Z^d}$ with stretched-exponential tails. Since $((X_{v,i})_{0 \le i < M})_{v \in \Z^d}$ is a finite-valued \iid\ process, this yields the required coding for $\omega^*$.
\end{proof}

Our method of proof of Proposition~\ref{prop:coding-simple} gives a slightly stronger result. We call $\sigma$ a \emph{simple stopping-process} if it is a $B^*$-stopping-process, where $B^*$ is defined by $B^*_n:=\{\zero\} \times \{0,1,\dots,n\}$. In this case, $X^\sigma$ can unambiguously be thought of as $(X_{v,i})_{v \in \Z^d, 0 \le i \le \sigma_v}$.

\begin{proposition}\label{prop:coding}
Let $X=(X_{v,i})_{v \in \Z^d, i \ge 0}$ be a finite-valued \iid\ process, let $B$ be linear, let $\tau$ be a stationary $B$-stopping-process for $X$ and $\sigma$ a stationary simple stopping-process for $X$. Suppose $\tau_v$ has exponential tails and $\E |B_{\tau_v}| < \E \sigma_v+1$. Then $X^\tau$ is a finitary factor of $X^\sigma$ with stretched-exponential tails.
\end{proposition}

Note that, since $\sigma$ is simple, the condition $\E |B_{\tau_v}| < \E \sigma_v+1$ may be more naturally written as $\E |B_{\tau_v}| < \E |B^*_{\sigma_v}|$.
Proposition~\ref{prop:coding-simple} is the special case of Proposition~\ref{prop:coding} in which $\sigma$ is taken to be the deterministic simple stopping-process given by $\sigma_v=M-1$ for all $v$. The rest of the paper is devoted to the proof of Proposition~\ref{prop:coding}.

\begin{remark}
	One may make slight modifications to the proof of the proposition to obtain various improvements. For instance, the same conclusion holds under the weaker assumptions that $\tau_v$ has only stretched-exponential tails and that $\Delta_n$ grows polynomially fast in $n$. In fact, one could even allow somewhat heavier tails and faster growing $\Delta_n$ at the expense of obtaining a coding radius with heavier tails. This is true even to the extent that, with no assumptions on the tails of $\tau_v$ or on the growth of $\Delta_n$, the conclusion still holds albeit with no information on the coding radius. On the other hand, under the stronger assumption that $\tau$ is also a simple stopping-process, the coding radius can be shown to have exponential tails.
\end{remark}

\begin{remark}
Proposition~\ref{prop:coding} holds also for \emph{random} simple stopping-processes~$\sigma$ (though we do not allow randomness in the $B$-stopping-process $\tau$), provided the randomness is made independent for each vertex in the following sense: There exists an \iid\ process $X'=(X'_v)_{v \in \Z^d}$, independent of~$X$, such that, for each $v$, $\sigma_v$ is an almost surely finite stopping time with respect to the filtration $(\cF^n_v \vee \sigma(X'_v))_{n \ge 0}$, where $\cF^n_v \vee \sigma(X'_v)$ is the smallest $\sigma$-algebra containing $\cF^n_v$ and the one generated by $X'_v$. By working conditionally on $X'$, the proofs go through essentially unchanged.
\end{remark}

\section{The algorithm}\label{sec:algo}

In this section, we provide the algorithm used to construct the finitary coding stated in Proposition~\ref{prop:coding}. We then use it in Section~\ref{sec:proof} to prove the proposition.

Throughout this section, we work in the setting of Proposition~\ref{prop:coding} so that $X=(X_{v,i})_{v \in \Z^d, i \ge 0}$ is an \iid\ process taking values in a finite set $S$, $B$ is linear, $\tau$ is a stationary $B$-stopping-process for $X$ and $\sigma$ is a stationary simple stopping-process for $X$. In addition, $\tau_v$ has exponential tails and $\E |B_{\tau_v}| < \E \sigma_v+1$.
We may also assume without loss of generality that $\sigma_v$ is bounded.

We construct an algorithm which, given a realization $\mathcal{Y}$ of the ``source'' process $X^\sigma$, deterministicly computes an output $\cZ$ having the distribution of the ``target'' process $X^\tau$. In the special case where $\sigma_v=M-1$ deterministically for all $v$ (as in Proposition~\ref{prop:coding-simple}), we could imagine that there is a single space-time landscape, initially containing variables in the subset $\Z^d \times \{0,1,\dots,M-1\}$ of space-time, and that these variables may be ``transported'' from their original locations to new locations as needed to construct $\cZ$. However, in general, as $\sigma$ is a stopping-process, we do not know which subset of space-time initially contains variables unless we expose some of the variables, but this would bias them and so we could not easily use them to construct $\cZ$.
Thus, instead of revealing the entire random field $\cY$ at once, the algorithm slowly reveals more and more of $\cY$ as is needed to generate more and more of~$\cZ$. As both $\cY$ and $\cZ$ are realizations of stopping-processes, it is convenient to think that the input to the algorithm is in fact a realization of $X$, which the algorithm uses to simultaneously construct both $\cY$ and $\cZ$ in such a manner that the variables of $X$ used to construct $\cZ$ are a subset of those used to construct $\cY$.
We may thus imagine that there are in fact three space-time landscapes: one corresponding to the original process $X$, one to the source process $\cY$, and one to the target process $\cZ$. When the algorithm wishes to reveal an additional piece of $\cY$, the required variable is easily generated -- it is simply read from the same location in the $X$ process. On the other hand, when an additional piece of $\cZ$ needs to be generated, it must be matched to a variable used by $\cY$. Here comes into play the crucial assumption that $\E |B_{\tau_v}| < \E \sigma_v+1$, which ensures that $\cZ$ uses less variables than $\cY$ on average. Thus, from the point of view of $\cZ$, as any variable used by $\cY$ is ``available'' to be used by $\cZ$, there are many available variables (much more than needed) for $\cZ$, and one needs only to find a suitable way of ``transporting'' these from the source to the target.

We call the variables of the source process $\cY$ \emph{inputs} and the variables of the target process $\cZ$ \emph{outputs}.
We stress that transporting a variable from $(u,i)$ to $(w,j)$ simply means that the source location $(u,i)$ and target location $(w,j)$ are matched to one another so that the input $\cY_{u,i}$ and the output $\cZ_{w,j}$ are identified.
Moreover, when we say that an input is generated, say at location $(u,i)$, we simply mean that the corresponding variable $X_{u,i}$ is revealed and identified with $\cY_{u,i}$, and when we say that an output is generated, say at location $(u,i)$, we mean that a suitable input is transported to $(u,i)$.

The algorithm consists of a ``simulator'' for each vertex $v \in \Z^d$, which has an associated \emph{source location} (thought of as a space-time location of $\cY$) and \emph{target location} (thought of as a space-time location of $\cZ$). At each time step $n$, the simulators simultaneously execute a common procedure (this will guarantee that any output of the algorithm is translation-equivariant).
The goal of the $v$-simulator is to ensure that its stopping time $\tau_v$ is reached (with respect to the target process $\cZ$) and that all relevant outputs for $\cZ_v$ (i.e., those corresponding to space-time locations in $v+B_{\tau_v}$) have been generated (that is, to determine an integer $t_v \ge 0$ and a configuration $\xi \in S^{v+B_{t_v}}$ for which $\tau_v(\xi)=t_v$). Once this happens, $v$ will be ``satisfied'', the final output $\cZ_v$ will be known, and the $v$-simulator will remain idle; until then, the $v$-simulator will be in a constant state of searching (in that its source location will change at every time step), trying to find an unused input at the source location which it can transport to the target location.
We note that there is a complex interplay between the different simulators. On the one hand, they are competing for shared resources, namely, the inputs. On the other hand, as different sites $v$ may rely on common outputs in order to compute their final output $\cZ_v$, the simulators may occasionally ``unintentionally help'' each other reach their goals (as long as it helps them too) by generating an output which is also required by another simulator (though we do not exploit this in the proof). This is in fact the origin of some complications, which presumably cannot be avoided.
Our algorithm is inspired partly by the algorithms in~\cite{van1999existence,harvey2006universal} (see Section~\ref{sec:alg-comparison} for a comparison between our algorithm and the one in~\cite{van1999existence}).

\subsection{Informal description of the algorithm}
\label{sec:alg-informal}
The goal of the $v$-simulator is to make sure that the final output $\cZ_v$ becomes known after some finite number of steps. To do this, the $v$-simulator proceeds as follows: Initially, at time step $n=0$, it reveals the variable $X_{v,0}$, which corresponds to the single space-time location in $v+B_0 = \{(v,0)\}$ (see Section~\ref{sec:formal-algo-def} for a formal definition of sets of the form $v+A$). It then consults the stopping rule $\tau_v$ to see whether or not it should continue. If it has reached the stopping time, i.e., $\tau_v = 0$, then the final output is known, namely, $\cZ_v$ is the element in $S^{B_0}$ given by $(\cZ_v)_{(v,0)}=X_{v,0}$, so that the simulator is satisfied and can stop.
If it has not reached the stopping time, i.e., $\tau_v > 0$, its next goal becomes to generate the outputs in $v+(B_1 \setminus B_0)$. Let us come back to how this is done in a moment. Once these have been generated (which may require many steps of the algorithm), the $v$-simulator consults the stopping rule $\tau_v$ again, this time to check whether $\tau_v = 1$. If indeed $\tau_v = 1$, then it is satisfied and the final output is known, namely, $\cZ_v$ is an element in $S^{B_1}$ given by the generated outputs at space-time locations $v+B_1$. If instead $\tau_v>1$, the $v$-simulator continues in a similar manner, with the general rule being that once the $v$-simulator learns that $\tau_v>k$, it continues to generate the outputs in $v + (B_{k+1} \setminus B_k)$, and then to check the stopping rule in order to determine whether or not it should continue. Eventually the stopping time is reached, the final output is known and the $v$-simulator is satisfied.

Let us now explain how the $v$-simulator generates the outputs in $v + (B_k \setminus B_{k-1})$. Firstly, it does so one output at a time (in an arbitrary order), and so we merely focus on how it generates a single output at space-time location $(w,j)$. Of course, one way to do this is simply to use the original variable residing at that location, namely, $X_{w,j}$. However, since we want to obtain a coding from $X^\sigma$, we must be sure to only use inputs (those variables residing in the scope of the source process), i.e., we cannot use $X_{w,j}$ unless $\sigma_w \ge j$. We also cannot use an input if it has already been used (transported away) by some other simulator at a previous time. Thus, we may need to search for an input at a different location $(u,i)$ and transport it from there to $(w,j)$.
Roughly speaking, the simulator moves along the space-time landscape of the source process, checking to see whether there is an unused input which it can transport to the target location $(w,j)$. At every time step, it checks a single source location $(u,i)$. If the input at that location is not available for use, the simulator simply advances its current source location, and does nothing further in that step of the algorithm. This procedure is repeated until the simulator eventually finds an unused input that it can transport. At that time, assuming the required output has not meanwhile been generated by another simulator, it transports it. Either way, the output at $(w,j)$ is sure to have been generated by the end of that step.

Of course, as we are trying to construct a coding, the above procedure must be carried out simultaneously by all the simulators. This leads to some interaction between the different simulators. Let us now give some more specific details about this and the above procedure.
We first explain how the $v$-simulator behaves with regards to the source process in each step:
\begin{itemize}
 \item If the simulator is satisfied, it does nothing. If it is unsatisfied, it will necessarily move its source location and it does so as follows. It first tries to move up one step in the pile of the vertex $u$ it is currently at. If it cannot, i.e., if it is already at the top of an exhausted pile (in the sense that the stopping time $\sigma_u$ has been reached), then it moves to the bottom of the pile located one step to the right of $u$ (i.e., to $u+e_1$). Here we informally refer to the inputs at locations $(u,i)$ as the \emph{pile} at $u$, and think of the pile there as initially empty and then growing as inputs there are generated until it becomes exhausted (i.e., until it reaches its full size given by the stopping time $\sigma_u$). 
 \item The above choice implies that if the $v$-simulator is at the top of a pile which has not yet been exhausted (we shall later call such a pile \emph{loaded}), then the input just above the top of that pile has not yet been used/revealed by any simulator. Thus, it is an unbiased input (having the same distribution as $X_{\zero,0}$) and is available to be transported. In this situation, regardless of whether or not it is indeed transported, the source location is moved one step up the pile.
 \item We initially set the $v$-simulator's source location to be $(v,-1)$ so that it is necessarily at the top of a loaded pile when the algorithm starts.
 \item Let us point out that when the pile sizes are deterministically fixed (as in the situation of Proposition~\ref{prop:coding-simple}), the evolution of the source location is also deterministic (up to knowing at what time the simulator becomes satisfied and stops). However, in general, as $\sigma$ is a simple stopping-process, the evolution is random: to decide whether or not a pile is exhausted, we must inspect the variables in the pile.
 \item We could have chosen different conventions here. Our choice has the advantage that there cannot be more than one unsatisfied simulator at any location at any given time. This means that we do not need to worry about different simulators trying to transport the same input.
\end{itemize} 
Next, we explain how the $v$-simulator behaves with regards to the target process:
\begin{itemize}
  \item If the simulator is satisfied, it does nothing. If it is unsatisfied, it might move its target location and it might not. Specifically, it moves precisely when its source is at the top of a loaded pile. Indeed, when this happens, we are assured that the required output can be generated. Moreover, when it moves, it moves to the next element in $v+B_\infty$, where the elements of $B_\infty=\bigcup_{n \ge 0} B_n$ are ordered in any way which respects the inclusions $B_0 \subset B_1 \subset \cdots$.
  \item Note that the times at which the target location changes is completely determined by the source. In particular, even if the output at the target location has been previously generated by some other simulator, this does not mean that the $v$-simulator will necessarily advance its target location. In other words, the output at the target location may have already been generated, and it may take the simulator many more steps until it finds an unused input (i.e., its source is at the top of a loaded pile), only to realize at that point in time that it does not need it after all (in which case that input will be wasted -- it will not be transported later). This is not the most efficient choice, but it is the one we make.
  \item We point out that, unlike for the source, there may be many different simulators at a given target location at the same time. This situation just means that the different simulators all wish to generate the same output.
  Among these simulators, many may also be at the top of a loaded pile (in the source), which means that they can transport an input. Thus, we must take care that different simulators do not generate the same output. We must therefore prioritize the simulators in some manner. To this end, we simply make the choice that the lexicographical-minimal simulator (among those at the top of a loaded pile) takes priority, namely, it is the one to generate the output, while the others do not transport an input (note that this is again not the most efficient way to do things, since we are throwing away inputs which could have been used later, but this is not too wasteful and we simply made a choice which we found convenient).
\end{itemize}

We emphasize that the algorithm may transport an input away from a certain location at some point in time, and then transport some other input into that same location at a later point in time.
That is, even if eventually there is an input at location $(u,i)$ (in the sense that $\sigma_u \ge i$) and the output at that same location is eventually needed by some simulator (in the sense that $(u,i) \in v+B_{\tau_v}$ for some $v$), there is no guarantee that the variable that will eventually end up to be the output at $(u,i)$ is the one that was originally the input there.
The important property is that any given input can only be transported away once, and any given output can only be generated (i.e., transported into) once. This is another reason it is helpful to imagine separate space-time landscapes for the source and target processes.

\smallskip
We refer the reader to Figure~\ref{fig:algorithm} for an illustration of the algorithm.

\begin{figure}
	\centering
	\captionsetup{width=0.92\textwidth}
	\includegraphics[scale=0.22]{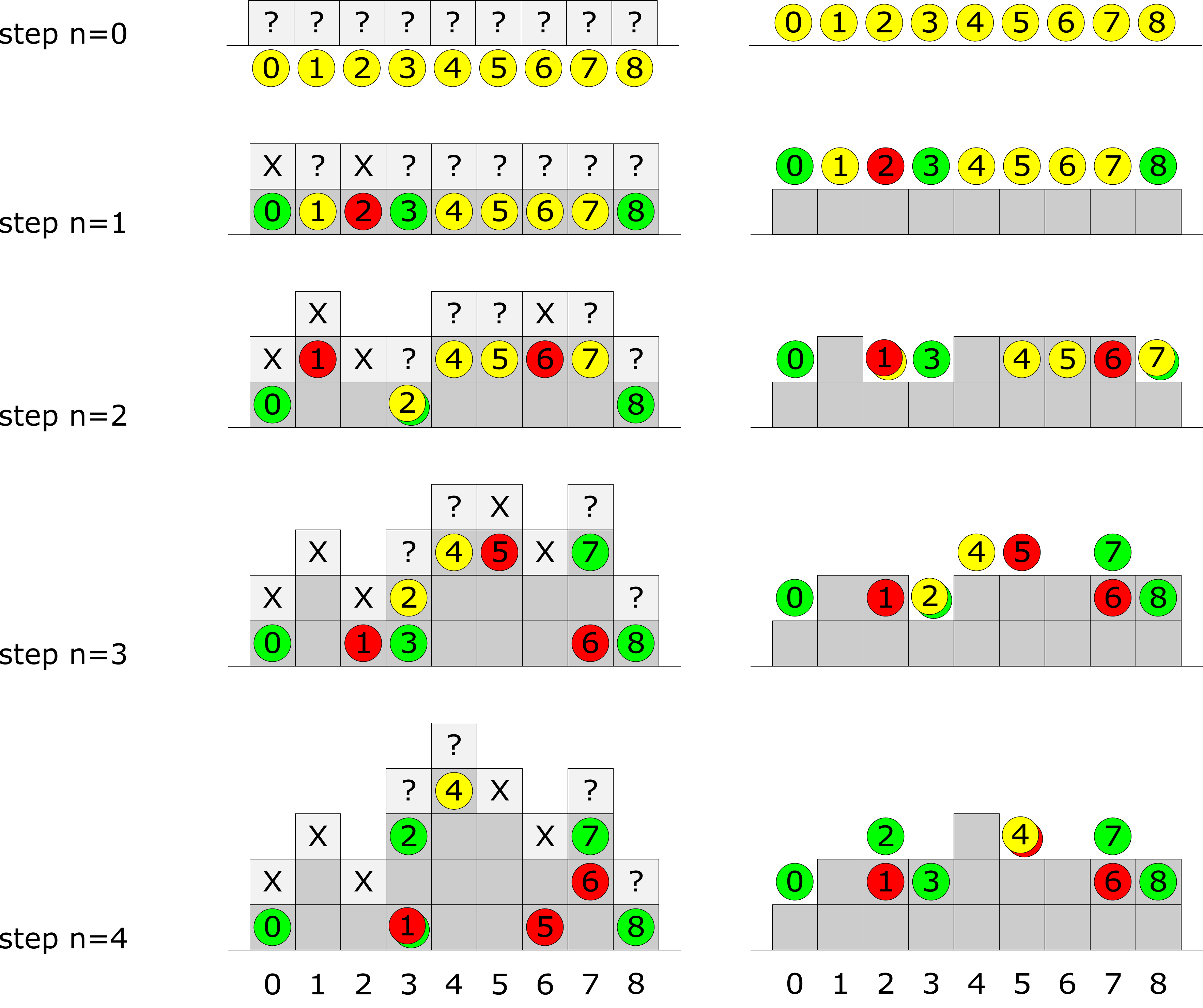}
	\caption{An illustration of the first five steps of the algorithm. For illustration purposes, we consider here the case where $d=1$ and $B_n=\{(u,i) : 0 \le u \le i \le n\}$, ordering the elements of $B_\infty$ as $(0,0),(0,1),(1,1),(0,2),(1,2),(2,2)$ and so on.
 The figure depicts the processes $Y^n$ and $Z^n$ and the state of the simulators at the end of step $n$ for $n=0,\dots,4$. The row just above the horizontal axis corresponds to the portion $\Z \times \{0\}$ of space-time.\vspace{2pt}\\
	\textit{Left:} The source process $Y^n$ and the source locations $(U^n_v,I^n_v)$ of the simulators. A gray background at space-time location $(u,i)$ indicates that the input $Y^n_{u,i}$ has been generated. An~$\times$ indicates an unloaded vertex, while a question mark indicates a loaded vertex.\vspace{2pt}\\
	\textit{Right:} The target process $Z^n$ and the target locations $(W^n_v,J^n_v)$ of the simulators. A gray background at space-time location $(w,j)$ indicates that the output $Z^n_{w,j}$ has been generated.\vspace{2pt}\\
	\textit{Simulators:} The simulators are depicted in green, yellow or red according to whether they are satisfied, unsatisfied but at the top of a loaded pile, or otherwise. A green simulator does not move as it has finished running (case (i) in the algorithm). A yellow simulator advances its source location by moving up one step in its current pile, reads the unused input at that new location, transports this input to its current target location (if needed), and then advances its target location by moving to the ``next place in line'' according to the ordering on $B_\infty$ (case (iv) in the algorithm). We note that when two yellow simulators occupy the same target location, only one of them actually generates the output (i.e., transports an input to that location). A red simulator does not have access to an unused input, and so it advances its source location by either moving up the current pile if it is not yet at the top (case (ii) in the algorithm) or otherwise by moving to the bottom of the next pile (case (iii) in the algorithm), while its target location remains unchanged.
In particular, a red simulator does not advance its target location even if the corresponding output is (or was previously) generated by a different simulator.\\See Section~\ref{sec:figure-explain} for further details about the figure.}
	\label{fig:algorithm}
\end{figure}

\subsection{Further explanation of the figure}
\label{sec:figure-explain}
Figure~\ref{fig:algorithm} illustrates the first several steps of the algorithm.
The figure contains a detailed caption, and here we provide some additional information.

Let us first address the setting considered in the figure. Of course we consider $d=1$ as it would be difficult to provide a useful picture for two of more dimensions. On the other hand, the specific $B_n$ considered there is not essential, and the reason for that choice was to allow the simulators to ``climb up'' in a short number of steps. We note that this choice for $B_n$ may be regarded as a simplification of what would be used for the one-dimensional case of Theorem~\ref{thm:main} (since the $B_n$ are only ``one-sided cones'', whereas the theorem would require symmetric ``two-sided cones'').

Let us now consider the evolution of the simulators throughout the steps depicted in the figure. Initially, the source and target locations of each $v$-simulator are set to $(v,-1)$ and $(v,0)$, respectively. This means that $v$-simulator is currently trying to generate the output at location $(v,0)$ and it is currently looking for an unused input (which it would like to transport) just above the source location $(v,-1)$, namely, at $(v,0)$. Indeed, initially there is always an unused input there (since $\sigma_v \ge 0$ by assumption). This situation is depicted at the top of the figure. Thus, at step $n=1$ of the algorithm, every $v$-simulator moves its source location one step up the pile to $(v,0)$, (vacuously) transports the input from $(v,0)$ to $(v,0)$, and advances its target location to $(v,1)$ (note that $(0,1)$ is the successor of $(0,0)$ in the chosen ordering of $B_\infty$). At this stage, some simulators have already become green (satisfied) and thus have $\tau_v=0$ -- these are simulators 0, 3 and 8. Let us follow what happens next to simulator 1 (which is still unsatisfied): since the 1-simulator is yellow (it is at the top of a loaded pile), it moves its source location one step up the pile to $(1,1)$, (vacuously) transports the input from $(1,1)$ to $(1,1)$, and advances its target location to $(2,1)$ (because $(1,1)$ follows $(0,1)$ in the order on $B_\infty$). Since at the end of step 2, the 1-simulator is red (it is no longer in a loaded pile), in step 3 it does not change its target location and simply moves its source location to the bottom of the next pile, which is $(2,0)$.
Since it is still red, in step 4 it again only moves its source location, this time to $(3,0)$.
We stress that even though, at the end of step 3, the output at the 1-simulator's target location $(2,1)$ has already been generated (it was transported from location $(3,0)$ by the 2-simulator in step 3), the 1-simulator still does not advance its target location; it will only do so once it becomes yellow.
Finally, since the 1-simulator is still red at the end of step 4, in the next step (which is not depicted in the figure) it will move its source location one step up the pile to $(3,1)$. At this stage, we still do not know the eventual value of $\tau_1$, we only know that $\tau_1 \ge 1$ (since there is an output in $1+B_1$ which is needed). Similarly, at the end of step~4, we know that $\tau_0=\tau_3=\tau_8=0$, $\tau_2=\tau_7=1$, $\tau_6 \ge 1$, $\tau_4 \ge 2$ and $\tau_5 \ge 2$. In particular, we know the final output for vertices $\{0,2,3,7,8\}$, but not yet for $\{1,4,5,6\}$.

\subsection{Formal definition of the algorithm}\label{sec:formal-algo-def}
Before providing the algorithm, we require some preparation.

Let us employ the following useful convention regarding stopping times.
Suppose that $\pi$ is an almost surely finite stopping time with respect to the filtration $(\cF^n_\zero)_{n \ge 0}$ defined in~\eqref{eq:F_n}.
We may regard $\pi$ as a deterministic function from $\bigcup_{n \ge 0} S^{B_n}$ to $\N \cup \{*\}$ having the property that, for any $n \ge 0$, $\xi \in S^{B_n}$ and $\xi' \in S^{B_{n+1}}$ such that $\xi'|_{B_n}=\xi$, we have $\pi(\xi) \in \{0,\dots,n,*\}$, we have $\pi(\xi')=\pi(\xi)$ when $\pi(\xi)\neq *$, and we have $\pi(\xi') \in \{n+1,*\}$ when $\pi(\xi)=*$.
The interpretation here is that a value of $*$ means that the stopping time has not been reached. Note, in particular, that for $m \ge 0$ and $\eta \in S^{B_{n+m}}$, the expression $\pi(\eta)>n$ depends only on $\eta|_{B_n}$ (where it is understood that $*>n$ for all integer $n$).
With this in mind, we note that (with a slight abuse of notation), if $A \subset \Z^d \times \N$ contains $B_n$, then the expression $\pi(\eta) \le n$ is well-defined for any $\eta \in S^A$ and depends only on $\eta|_{B_n}$, and thus, the expression $b \in B_{\pi(\eta)}$ is also well-defined for any $b \in B_{n+1}$ (and depends only on $\eta|_{B_n}$).
We further abuse notation by identifying an element $\eta \in (S \cup \{\emptyset\})^{\Z^d}$ with the element $\eta' \in S^A$ in the obvious way, by taking $A:=\{a \in \Z^d \times \N : \eta(a) \neq \emptyset \}$ and $\eta':=\eta|_A$.

We order the elements of $B_\infty := \bigcup_{n \ge 0} B_n$ in such a manner that, for any $n$, every element of $B_n$ appears before every element of $B_\infty \setminus B_n$. This induces a notion of \emph{successor} for elements in $B_\infty$. For $A \subset \Z^d \times \N$, we write $v+A$ for the translated set $\{(v+u,i) : (u,i) \in A\}$. Translating the ordering from $B_\infty$ to $v+B_\infty$, we obtain a notion of \emph{$v$-successor} for elements in $v+B_\infty$. More precisely, the $v$-successor of an element $(v+u,i) \in v+B_\infty$ is $(v+u',i')$, where $(u',i')$ is the successor of $(u,i)$.

At each step $n \ge 0$, we define variables:
\begin{itemize}
	\item $(U^n_v,I^n_v) \in \Z^d \times \N$, the source location of the $v$-simulator.
	\item $(W^n_v,J^n_v) \in \Z^d \times \N$, the target location of the $v$-simulator.
	\item $T^n_v \in \{0,1\}$, the indicator of whether the $v$-simulator transported input (generated output).
\end{itemize}
Once the above variables are defined at step $n$, we further define several objects, all of which are deterministic functions of the above variables. For some of these definitions to make sense, it is important to note that the following properties are satisfied at every step $n$:
\begin{align}
\sum_v \sum_{t=1}^n \1_{T^t_v=1,(U^t_v,I^t_v)=(u,i)} \le 1 &\qquad\text{for all }(u,i) \in \Z^d \times \N,
\label{eq:once-read}\\
	\sum_v \sum_{t=1}^n \1_{T^t_v=1,(W^{t-1}_v,J^{t-1}_v)=(w,j)} \le 1 &\qquad\text{for all }(w,j) \in \Z^d \times \N. \label{eq:once-written}
\end{align}
Equation~\eqref{eq:once-read} says that each input is transported away at most once by at most one simulator. Similarly, \eqref{eq:once-written} says that every output is generated (transported into) at most once by at most one simulator. As the target location of the $v$-simulator will be updated immediately after the required output is generated, there is a shift in the time index in~\eqref{eq:once-written}. Thus, $T^n_v=1$ means that at time step $n$ the $v$-simulator transported an input from the source location $(U^n_v,I^n_v)$ to the target location $(W^{n-1}_v,J^{n-1}_v)$, thus generating the output at $(W^{n-1}_v,J^{n-1}_v)$.

Consider the set of source-target locations of simulators at transport times:
\[ D^n := \left\{ (v,u,i,w,j) : \substack{\text{\normalsize $v \in \Z^d,~T^t_v=1,~(U^t_v,I^t_v)=(u,i),$}\\\text{\normalsize $(W^{t-1}_v,J^{t-1}_v)=(w,j)\text{ for some }1 \le t\le n$}} \right\}.\]
We use $D^n$ and $X$ to construct two $(S \cup \{\emptyset\})$-valued processes $Y^n=(Y^n_{u,i})_{u\in\Z^d,i\ge0}$ and $Z^n=(Z^n_{w,j})_{w\in\Z^d,j\ge0}$, which represent the partial information on $\mathcal{Y}$ and $\cZ$ (the realizations of $X^\sigma$ and $X^\tau$) that has been revealed by time $n$.
The algorithm may use one of two ``update methods'': for $(v,u,i,w,j) \in D^n$, we define
\[ \textbf{(A)}~ Y^n_{u,i} = Z^n_{w,j} := X_{u,i}, \qquad\textbf{(B)}~ Y^n_{u,i} = Z^n_{w,j} := X_{w,j} .\]
If $(u,i)$ is not in the projection of $D^n$ on the 2nd and 3rd coordinates, then set $Y^n_{u,i} := \emptyset$, and similarly, if $(w,j)$ is not in the projection of $D^n$ on the 4th and 5th coordinates, then set $Z^n_{w,j} := \emptyset$. Note that~\eqref{eq:once-read} and~\eqref{eq:once-written} ensure that both update methods are well-defined.
We stress that the two update methods are never used in conjunction with one another -- either update method \textbf{(A)} is used throughout all steps of the algorithm or update method \textbf{(B)} is.

The $Y$ process is associated with $\sigma$, and the $Z$ process with $\tau$. Thus, in update method \textbf{(A)}, $\sigma$ ``sees'' the original process $X$, while $\tau$ ``sees'' a transformed process in which inputs have been transported between space-time locations; in update method \textbf{(B)}, the situation is reversed -- $\tau$ sees the origin process and $\sigma$ sees a transformed process.
Another point of view is that $D^n$ (after forgetting the first coordinate) defines a bipartite graph between two copies of $\Z^d \times \N$ in which any vertex of one copy is matched to at most one vertex in the other copy. The two update methods can then be thought of as orienting all edges from the first copy to the second, or vice versa, where the orientation of an edge determines the direction of flow of information, with the original process $X$ always associated with the copy from which the edges are oriented outwards (so that variables are transported along the edges in the direction of orientation).
As we are interested in realizing the $\tau$ process via the $\sigma$ process, the natural choice is to transport variables from the latter to the former as in update method \textbf{(A)}. Nevertheless, it will turn out to be a helpful idea to consider also the reversed direction of flow.
Thus, update method \textbf{(A)} will yield the required coding, whereas update method \textbf{(B)} will only be used as a comparison tool in the analysis (namely in the proof of Lemma~\ref{lem:eventually-satisfied}). As such, we mainly have update method \textbf{(A)} in mind in our definitions.

We further define
\begin{itemize}
	\item $L^n_u := \max\{ i : (v,u,i,w,j) \in D^n\text{ for some }v,w,j \}$, the last input revealed at $u$.
	\item $u$ is \emph{loaded} at time $n$ if $\sigma_u(Y^n) > L^n_u$.
	\item $v$ is \emph{satisfied} at time $n$ if $(W^n_v,J^n_v) \notin v+B_{\tau_v(Z^n)}$.
\end{itemize}
Thus, $L^n_u$ is the size of the pile at $u$ (in the source process) at time $n$.
A vertex $u$ is loaded at time $n$ if there are more inputs available at $u$ than have already been used by time $n$, i.e., if the pile at $u$ has not been exhausted by time $n$. A vertex $v$ is satisfied at time $n$ if the output at the target location of the $v$-simulator is not needed in order to compute the final output $\cZ_v$. In particular, due to way that the target location evolves, this implies (but is not precisely equivalent to) that the outputs that $v$ needs for its final output have already been generated by time $n$ (see~\eqref{eq:previous-target-locations-are-known} below), so that the final output is known at this time.

The fact that the notions of loaded and satisfied are well-defined is not obvious from their definitions. The fact that the notion of loaded is well-defined follows from the above discussion about stopping times and the following property which will hold at each step $n$:
\begin{equation}\label{eq:previous-source-locations-are-known}
Y^n_{u,i} \neq \emptyset\qquad\text{if and only if}\qquad i \le L^n_u, \qquad\text{for all }(u,i) \in \Z^d \times \N.
\end{equation}
Similarly, the fact that the notion of satisfied is well-defined follows from the following property, which will hold for all $n$:
\begin{equation}\label{eq:previous-target-locations-are-known}
Z^n_{w,j} \neq \emptyset \qquad\text{for all $(w,j) \in \Z^d \times \N$ that strictly $v$-precede }(W^n_v,J^n_v)\text{ for some }v \in \Z^d .
\end{equation}

Finally, for $(w,j) \in \Z^d \times \N$, we also define
\[ Q^n(w,j):=\left\{ v : (W^n_v,J^n_v)=(w,j),~ \substack{\text{\normalsize $U^n_v$ is loaded at time $n$,~ $I^n_v = L^n_{U^n_v}$}\\\text{\normalsize $v$ is unsatisfied at time $n$,~ $Z^n_{w,j} = \emptyset$}} \right\} .\]
Thus, $Q^n(w,j)$ consists of those simulators who both wish to generate the output at $(w,j)$ and can also do so (they wish to do so as they are unsatisfied, meaning that they need that output, and as the output has not yet been generated; they can do so as they are at the top of a loaded pile in the source process). Since only one such simulator can be allowed to actually generate the output at $(w,j)$, we will let the lexicographical-minimal one do so.

\medskip

With these definitions, we can now present the algorithm. We refer the reader to Section~\ref{sec:alg-informal} for an informal description and to Figure~\ref{fig:algorithm} for an illustration.

\begin{algorithm}
\renewcommand\algorithmicindent{1.5em}
\renewcommand\thealgorithm{}
\caption{Finitary coding from $X^\sigma$ to $X^\tau$}
\begin{algorithmic}
\FOR{$v \in \Z^d$ (simultaneously)}
\STATE \vspace{3pt} $(U^0_v,I^0_v,W^0_v,J^0_v,T^0_v) \gets (v,-1,v,0,0)$ \vspace{3pt}
\ENDFOR
\FOR{$n=1,2,\dots$}
\FOR{$v \in \Z^d$ (simultaneously)}
\IF{$v$ is satisfied at time $n-1$}
\STATE \vspace{3pt}
\hspace{-12ex}\makebox[12ex][s]{i.}%
$(U^n_v,I^n_v,W^n_v,J^n_v,T^n_v) \gets (U^{n-1}_v,I^{n-1}_v,W^{n-1}_v,J^{n-1}_v,0)$ \vspace{3pt}
\ELSIF{$I^{n-1}_v < L^{n-1}_{U^{n-1}_v}$}
\STATE \vspace{3pt}
\hspace{-12ex}\makebox[12ex][s]{ii.}%
$(U^n_v,I^n_v,W^n_v,J^n_v,T^n_v) \gets
	(U^{n-1}_v,I^{n-1}_v+1,W^{n-1}_v,J^{n-1}_v,0)$ \vspace{3pt}
\ELSIF{$U^{n-1}_v$ is unloaded at time $n-1$}
\STATE \vspace{3pt}
\hspace{-12ex}\makebox[12ex][s]{iii.}%
$(U^n_v,I^n_v,W^n_v,J^n_v,T^n_v) \gets (U^{n-1}_v+e_1,0,W^{n-1}_v,J^{n-1}_v,0)$ \vspace{3pt}
\ELSE
\STATE \vspace{3pt}
\hspace{-12ex}\makebox[12ex][s]{iv.}%
$(U^n_v,I^n_v) \gets (U^{n-1}_v,I^{n-1}_v+1)$
\STATE $(W^n_v,J^n_v) \gets \text{$v$-successor of }(W^{n-1}_v,J^{n-1}_v)$
\STATE $T^n_v \gets \1(\text{$v$ is the lexicographical-minimal element of $Q^{n-1}(W^{n-1}_v,J^{n-1}_v)$)}$
\vspace{3pt}
\ENDIF
\ENDFOR
\ENDFOR
\end{algorithmic}
\end{algorithm}

\subsection{Comparison between our algorithm and that of van den Berg and Steif in~\cite{van1999existence}}
\label{sec:alg-comparison}
The two algorithms are similar in spirit (though they are not set up in the same way) and we focus here on the moral differences between the two. We have identified two such differences, the primary one being in how they relate to unneeded variables and, consequently, in how they transport such variables between space-time locations. Here, ``needed'' may refer to either an input or an output, where an input (output) at location $(u,i)$ is needed by time $n$ if $Y^n_{u,i} \neq \emptyset$ ($Z^n_{u,i} \neq \emptyset$). Roughly speaking, the algorithm in~\cite{van1999existence} declares an input variable unneeded at a certain time once it is guaranteed that the output variable at the same location will not be needed at any later time (and was also not needed until that time). Only inputs which are marked as unneeded in this sense are allowed to be transported. On the other hand, our algorithm never declares an input variable unneeded. Instead, we only concern ourselves with whether an input was not needed by a certain time, and any such variable is allowed to be transported at that time. If at a later time it turns out that the output at the same location was needed after all, another input variable will be transported to that location. In other words, the algorithm in~\cite{van1999existence} transports an input from location $(u,i)$ to another location $(w,j)$ only if the output at $(u,i)$ is never needed, whereas our algorithm does not have this restriction, and may transport from $(u,i)$ to $(w,j)$ at some time, and then from $(u',i')$ to $(u,i)$ at a later time. The latter approach is essential in the generality of Proposition~\ref{prop:coding-simple} and Proposition~\ref{prop:coding}. The reason is that, while for some choices of $B=(B_n)_n$, any particular output variable could only be potentially needed by finitely many vertices (e.g., as for the ``cones'' used in the proof of Theorem~\ref{thm:main}, where the output at $(w,j)$ can only be needed by vertices at distance at most $\Delta j$ from $w$), in general, any vertex might need that variable at some time (e.g., as for the ``cubes'' given by $B_n=\{ (u,i) : |u| \le \Delta n,\, 0 \le i \le n \}$) so that it is not possible to know (in a finitary manner) whether or not an output variable will be needed eventually. The second difference between the algorithms is that, unlike the algorithm in~\cite{van1999existence}, ours is somewhat wasteful (by design; see Section~\ref{sec:alg-informal}) in that in certain situations it decides not to use an available input variable (and to simply throw it away). We found this useful (though it is probably not essential) for keeping track of how far variables are transported, which was important for understanding the coding radius.

\section{Proof of Proposition~\ref{prop:coding}}
\label{sec:proof}

In this section, we use the algorithm described in Section~\ref{sec:algo} to prove Proposition~\ref{prop:coding}.

The following claim establishes some simple properties of the algorithm.
Let $\preceq$ denote the partial order on $\Z^d$ in which $u \preceq u'$ if $u'=u+ke_1$ for some $k \ge 0$.
We also denote by $\preceq$ the partial order on $\Z^d \times \Z$ in which $(u,i) \preceq (u',i')$ if $u \preceq u'$ and $(i'-i)\1_{u=u'} \ge 0$.

\begin{claim}\label{cl:props}
In either update method, almost surely, for all $n \ge 1$, $i,j \ge 0$ and $u,v,v',w \in \Z^d$,
\begin{enumerate}
	\item \eqref{eq:once-read}, \eqref{eq:once-written}, \eqref{eq:previous-source-locations-are-known} and \eqref{eq:previous-target-locations-are-known} hold.
	\item If $v \prec v'$ and both are unsatisfied at time $n-1$, then $(U^{n-1}_v,I^{n-1}_v) \prec (U^n_v,I^n_v) \prec (U^n_{v'},I^n_{v'})$.
	\item If $v \preceq w \prec U^n_v$, then $w$ is unloaded at time $n-1$.
\end{enumerate}
\end{claim}
\begin{proof}
The claim follows easily by induction on $n$.
\end{proof}

The following lemma states precisely the intuitive fact that transporting inputs from one space-time location to another does not change the resulting distribution. Denote the state at time $n$ by $\mathbb{S}^n := (U^n,I^n,W^n,J^n,T^n,D^n)$, where $U^n=(U^n_v)_{v \in \Z^d}$, $I^n=(I^n_v)_{v \in \Z^d}$ and so forth.

\begin{lemma}\label{lem:update-methods-equiv}
The distribution of $(\mathbb{S}^n,Y^n,Z^n)_{n \ge 0}$ does not depend on whether update method \textbf{(A)} or \textbf{(B)} is used in the algorithm.
\end{lemma}
\begin{proof}
	Observe that the algorithm does not explicitly depend on the update method used, but rather depends on it implicitly through the definitions of $Y^n$ and $Z^n$.
We prove by induction that the distribution of $\cS^n := (\mathbb{S}^m,Y^m,Z^m)_{0 \le m \le n}$ does not depend on the update method.
This is immediate for $n=0$, since $\cS^0$ is deterministic. Fix $n \ge 1$ and observe that $\mathbb{S}^n$ is measurable with respect to $\cS^{n-1}$. It thus suffices to show that (i) when using update method \textbf{(A)}, conditioned on $\cS^{n-1}$, $(X_{u,i})_{(v,u,i,w,j)\in D^n \setminus D^{n-1}}$ is a sequence of independent random variables having the distribution of $X_{\zero,0}$, and (ii) when using update method \textbf{(B)}, conditioned on $\cS^{n-1}$, $(X_{w,j})_{(v,u,i,w,j)\in D^n \setminus D^{n-1}}$ is such a sequence. Indeed, (i) follows easily from~\eqref{eq:once-read} and (ii) from~\eqref{eq:once-written}.
\end{proof}

\begin{lemma}\label{lem:eventually-satisfied}
Almost surely, every $v \in \Z^d$ is eventually satisfied. Moreover,
\[ \Pr(v\text{ is not satisfied at time }n) = e^{-\Omega\big(n^{1/(d+2)}\big)} \qquad\text{as }n \to \infty .\]
\end{lemma}

As will be explained in the proof of Proposition~\ref{prop:coding} below, Lemma~\ref{lem:eventually-satisfied} implies that the algorithm ``locally terminates'' in finite time in the sense that the final output at any vertex is determined at some finite step.
Nevertheless, this does not yet imply that the algorithm yields a finitary coding.
What is missing is some control on the propagation of information in each step. This is the content of the following lemma.
Let $\Delta$ be as in~\eqref{eq:linear-stopping-process}.
Denote $D^n_v := \{ (u,i,w,j) : (v,u,i,w,j) \in D^n \}$.

\begin{lemma}\label{lem:info-propagation}
When using update method \textbf{(A)}, for any $n \ge 0$ and $v \in \Z^d$, the following random variables are measurable with respect to $(X_{u,i})_{|u-v| \le 5\Delta n^2, 0 \le i \le \sigma_u}$:
\begin{enumerate}
	\item $\mathbb{S}^n_v = (U^n_v,I^n_v,W^n_v,J^n_v,T^n_v,D^n_v)$,
	\item $\{ Z^n_{w,j} \}_{|w-v| \le \Delta n, j \ge 0}$.
\end{enumerate}
\end{lemma}

Before proving Lemma~\ref{lem:eventually-satisfied} and Lemma~\ref{lem:info-propagation}, we first use them to prove the proposition.

\subsection{Proof of Proposition~\ref{prop:coding}}
	Consider the algorithm using update method \textbf{(A)}.
	For $n \ge 1$, define
	\[ M^n_v :=  \min \big\{ m \ge 0 : (W^n_v,J^n_v) \in v+B_m \big\} .\]
	Note that  $M^n_v \ge 1$ so that~\eqref{eq:previous-target-locations-are-known} implies that $Z^n_{w,j} \neq \emptyset$ for all $(w,j) \in v+B_{M^n_v-1}$. Thus,
	\[ \cZ^n_v := (Z^n_{v+w,j})_{(w,j) \in B_{M^n_v-1}} \]
	takes values in $\bigcup_{n \ge 0} S^{B_n}$.	
	Let $N_v$ denote the time at which $v$ is first satisfied.
	By Lemma~\ref{lem:eventually-satisfied}, $N_v$ is almost surely finite.
	Recall our conventions about stopping times discussed in the beginning of Section~\ref{sec:algo}.
	By the definition of satisfied, $(W^{N_v}_v,J^{N_v}_v) \notin v+B_{\tau_v(Z^{N_v})}$, so that $\tau_v(Z^{N_v}) < M^{N_v}_v$. Similarly, since $v$ is not satisfied at time $N_v-1$, it follows that $\tau_v(Z^{N_v-1}) \ge M^{N_v-1}_v$, and hence, also that $\tau_v(Z^{N_v}) \ge M^{N_v-1}_v$. Since $M^n_v \le M^{n-1}_v+1$, we conclude that $\tau_v(Z^{N_v}) = M^{N_v}_v - 1$.
	Thus, $\cZ^{N_v}_v = (Z^{N_v}_{v+w,j})_{(w,j) \in B_{\tau_v}}$, where $\tau_v := \tau_v(Z^{N_v}) = \tau_\zero(\cZ^{N_v}_v)$.
	Lemma~\ref{lem:update-methods-equiv} now implies that $(\mathcal{Z}^{N_v}_v)_{v \in \Z^d}$ equals $X^\tau$ in distribution. Since all the operations in the algorithm are translation-equivariant, we have thus obtained a coding from $X^\sigma$ to $X^\tau$.
	
	Let us check that this coding is finitary and that its coding radius $R$ has stretched-exponential tails.
	Indeed, since Lemma~\ref{lem:info-propagation} implies that $\{N_\zero \le n\}$ and $\cZ^n_\zero$ are measurable with respect to $(X_{u,i})_{|u| \le 5\Delta n^2, 0 \le i \le \sigma_u}$, it follows that $R \le 5\Delta N_\zero^2$. Lemma~\ref{lem:eventually-satisfied} then yields that
	\[ \Pr\big(R > 5\Delta n^2\big) \le \Pr\big(N_\zero > n\big) = \Pr\big(\zero\text{ is not satisfied at time }n\big) = e^{-\Omega\big(n^{1/(d+2)}\big)} . \tag*{\qed} \]

\subsection{Proof of Lemma~\ref{lem:eventually-satisfied}}

For the proof of Lemma~\ref{lem:eventually-satisfied}, we require a large-deviation-type result, which we now describe.
Let $X=(X_i)_{i \in \Z}$ be a sequence of non-negative random variables.
We say that $X$ is \emph{stopping-like} if there exists $\Delta>0$ such that for any finite $I,J \subset \Z$ and any non-negative numbers $(r_i)_{i \in I \cup J}$, the two events $\{ X_i > r_i\text{ for }i \in I \}$ and $\{ X_j > r_j\text{ for }j \in J \}$ are independent whenever the two sets $\bigcup_{i \in I} [i-\Delta r_i,i+\Delta r_i]$ and $\bigcup_{j \in J} [j-\Delta r_j,j+\Delta r_j]$ are disjoint.
Observe that, if there exists a sequence $(Y_i)_{i \in \Z}$ of independent random variables satisfying that, for any $i \in \Z$ and any $r \ge 0$, the event $\{X_i > r\}$ is measurable with respect to $\{Y_j\}_{|i-j|\le \Delta r}$, then $X$ is stopping-like.
Observe also that, if $X$ is a stopping-like process, then $(X_i \1_{\{X_i \le r\}})_{i \in \Z}$ is a $2\Delta r$-dependent process for any $r>0$, where a process $(Y_i)_{i \in \Z}$ is said to be $k$-dependent if $(Y_i)_{i \in I}$ and $(Y_j)_{j \in J}$ are independent whenever $I,J \subset \Z$ satisfy that $|i-j|>k$ for all $i \in I$ and $j \in J$.

\begin{lemma}\label{lem:stretched-exp-decay}
	Let $X=(X_i)_{i \in \Z}$ be a non-negative stopping-like stationary sequence and suppose that $X_0$ has exponential tails.
	Let $f \colon [0,\infty) \to [0,\infty)$ be a measurable function satisfying that $f(t) \le B t^b$ for some $B,b>0$ and all $t \ge 1$. Denote $\mu := \E f(X_0)$ and $\beta := \frac{1}{1+b}$.
	Then, for any $a>\mu$,
	\[ \Pr\big(f(X_1)+\cdots+f(X_n) \ge an\big) = e^{-\Omega(n^\beta)}\qquad\text{as }n \to \infty .\]
\end{lemma}
\begin{proof}
The proof involves a truncation argument. Thus, we first write $f(X_i)=Y_i+Y'_i$, where
\[ Y_i := f(X_i) \1_{\{X_i \le n^\alpha \}} \qquad\text{and}\qquad Y'_i := f(X_i) \1_{\{X_i>n^\alpha\}}, \]
and $\alpha$ is any positive number less than $\frac12 \beta(1-\beta)$.
Then, for any $\epsilon>0$,
\[ \Pr\big(f(X_1)+\cdots+f(X_n) \ge (\mu+2\epsilon)n\big) \le \Pr\big(Y_1+\cdots+Y_n \ge \mu n + \epsilon n\big) + \Pr\big(Y'_1+\cdots+Y'_n \ge \epsilon n\big) .\]
Thus, it suffices to bound separately the two terms on the right, showing that each is $e^{-\Omega(n^\beta)}$.

For the first term, we prove the stronger bound
\begin{equation}\label{eq:truncated_sum}
\Pr\big(Y_1+\cdots+Y_n \ge \mu n + \epsilon n\big) = e^{-\Omega(n^{1-2\alpha/\beta})} .
\end{equation}
For $0 \le i \le n$, denote
\[ Z_i := \E[Y_1+\cdots+Y_n \mid \cF_i], \qquad\text{where }\cF_i := \sigma\big(\big\{ X_j \1_{X_j \le n^\alpha}, \1_{X_j \le n^\alpha} \big\}_{1 \le j \le i}\big) .\]
Note that $(Z_i)_{0 \le i \le n}$ is a martingale satisfying $Z_0=(\E Y_0)n \le \mu n$ and $Z_n=Y_1+\cdots+Y_n$. Hence, by the Azuma--Hoeffding inequality~(see, e.g., \cite{williams1991probability}),
\[ \Pr\big(Y_1+\cdots+Y_n \ge \mu n + t\big) \le \Pr\big(Z_n \ge Z_0 + t\big) \le \exp\left(-\frac{t^2}{2\sum_{i=1}^n c_i^2}\right) , \qquad t \ge 0,\]
where $c_i :=\|Z_i-Z_{i-1}\|_\infty$ is the essential supremum of the increment $Z_i-Z_{i-1}$. Thus, \eqref{eq:truncated_sum} will follow if we show that $c_i \le Cn^{\alpha/\beta}$.
Indeed, since $Y$ is a $Bn^{\alpha b}$-bounded $2\Delta n^\alpha$-dependent process,
\[ |Z_i-Z_{i-1}|
 = \bigg|\sum_{j=i}^n \big(\E[Y_j \mid \cF_i] - \E[Y_j \mid \cF_{i-1}]\big)\bigg|
 \le\sum_{j=i}^{i+\lfloor 2\Delta n^\alpha \rfloor} \Big|\E[Y_j \mid \cF_i] - \E[Y_j \mid \cF_{i-1}]\Big| \le Cn^{\alpha/\beta} .\]

We now turn to the second term. Note that $\{Y'_1+\cdots+Y'_n \ge \epsilon n\} \subset E \cup F$, where
\[ \cI := \big\{ 1 \le i \le n : X_i > n^\alpha \big\}, \qquad E := \Big\{ |\cI| \ge (\epsilon n)^\beta \Big\} ,\qquad F := \Big\{ \max_{1 \le i \le n} X_i \ge \tfrac{1}{B} (\epsilon n)^\beta \Big\} .\]
For $I \subset \Z$, denote $d(I) := \min \{ |i-j| : i,j \in I,~ i \neq j\}$. Since, for any $I \subset \Z$ and integer $d \ge 1$, there exists a subset $I' \subset I$ such that $|I'| \ge |I|/d$ and $d(I') \ge d$, we obtain
\[ \Pr(E) \le \Pr\Big(\exists I \subset \cI,~|I|=\left\lceil\tfrac{(\epsilon n)^\beta}{2n^\alpha+2}\right\rceil,~d(I) \ge 2n^\alpha+1 \Big) . \]
Since the events $\{X_i>r\}_{i \in I}$ are independent for any finite $I \subset \Z$ and $0 \le r < d(I)/2$, we have
\[ \Pr(E) \le \binom{n}{\left\lceil\tfrac{(\epsilon n)^\beta}{2n^\alpha+2}\right\rceil} \cdot \Pr(X_0 > n^\alpha)^{\left\lceil\tfrac{(\epsilon n)^\beta}{2n^\alpha+2}\right\rceil} \le e^{\tfrac{(\epsilon n)^\beta}{n^\alpha} \log n-c(\epsilon n)^\beta} = e^{-\Omega(n^\beta)} . \]
Finally, it is immediate that $\Pr(F) \le n \cdot \Pr(X_0 \ge \tfrac{1}{B} (\epsilon n)^\beta) = e^{-\Omega(n^\beta)}$.
\end{proof}

\begin{remark}
	The bound in Lemma~\ref{lem:stretched-exp-decay} is tight, as the following simple example shows. Let $(Y_i)_{i \in \Z}$ be independent unbiased coin tosses, and let $X_i$ be the length of the streak of heads containing position~$i$, i.e., $X_i := \max\{k+m : Y_j=1\text{ for }i-k \le j < i+m,~k,m \ge 0 \}$. Clearly, $X$ is a stationary sequence (in fact, it is \ffiid\ with exponential tails) and $X_0$ has exponential tails. Moreover, since $X_i$ is a stopping time with respect to $(\{Y_j\}_{|j-i| \le n})_n$, it follows that $X$ is stopping-like (with $\Delta=1$). On the other hand, $\Pr(X_1^b+\cdots+X_n^b \ge an) \ge \Pr(Y_1=\cdots=Y_{\lceil (an)^\beta \rceil}=1)=2^{-\lceil (an)^\beta \rceil}$.
\end{remark}

\begin{proof}[Proof of Lemma~\ref{lem:eventually-satisfied}]
	As Lemma~\ref{lem:update-methods-equiv} implies that both update methods yield the same probability for the event in question, we may assume here that update method \textbf{(B)} is used in the algorithm.
	Denote $L^\infty_w := \sup_n L^n_w$. Let $\cL$ denote the vertices which remain loaded indefinitely.
	Let $v \in \Z^d$. For an integer $i$, we write $v+i$ for the element $v+ie_1 \in \Z^d$.
	Let us check that if, for some $k \ge 0$,
	\[ \cL_k := \{ 0 \le i \le k : v+i \in \cL\} \neq \emptyset \quad\text{or}\quad N_k := \sum_{i=0}^k (L^\infty_{v+i}+1) > \sum_{i=0}^k |B_{\tau_{v+i}(X)}| =: M_k ,\]
	then $v$ is satisfied at time $N_k$.
	Assume towards a contradiction that $v$ is not satisfied at time $N_k$. Then, by Claim~\ref{cl:props},
	\[ (v,-1)=(U^0_v,I^0_v) \prec (U^1_v,I^1_v) \prec \cdots \prec (U^{N_k+1}_v,I^{N_k+1}_v) .\]
	Thus, since $0 \le I^n_v \le L^\infty_{U^n_v}$ for all $n \ge 1$ by definition of $L^n_u$, we have $v+k \prec U^{N_k+1}_v$. In particular, the set of times
	\[ T := \{ 1 \le n \le N_k : U^n_v \prec U^{n+1}_v \} \]
	at which the $v$-simulator moved its source location to the right is of size $|T| \ge k+1$.
	Moreover, since, for $1 \le n \le N_k$, $n \in T$ if and only if case (iii) of the algorithm is executed at step $n$ for the vertex $v$, which in turn occurs only if $U^{n-1}_v \notin \cL$ and $I^{n-1}_v=L^{n-1}_{U^{n-1}_v}=L^\infty_{U^{n-1}_v}$, we conclude that
	\[ \cL_k = \emptyset,\qquad |T|=k+1, \qquad (U^{N_k+1}_v,I^{N_k+1}_v)=(v+k+1,0) .\]
	Note that if the input from some location $(u,i)$ is transported by some $v'$-simulator by time $n$ (i.e., $(v',u,i,w,j) \in D^n$ for some $(w,j)$), then $v' \preceq u$ and every $v''$ such that $v' \prec v'' \preceq u$ must be satisfied at time $n$.
	Thus, since by step $N_k+1$, $N_k$ inputs were transported from locations $(u,i)$ with $v \preceq u \preceq v+k$, but no more than $M_k$ inputs were transported by $v'$-simulators with $v \preceq v' \preceq v+k$, it follows that $v$ is satisfied at time $N_k$, which is a contradiction.
	Hence, $v$ is satisfied at time $N_k$.
	
	We have thus shown that
	\[ \Pr(v\text{ is not satisfied at time }n) \le \Pr\big(\forall k \ge 0 ~ (\cL_k=\emptyset\text{ and }N_k \le M_k)\text{ or }N_k > n\big) .\]
	Using that $L^n_v \le \sigma_v(Y^n) \le m$ almost surely for some $m \ge 1$, and taking $k=\frac{n}{4m}$, we get
	\[ \Pr(v\text{ is not satisfied at time }n) \le \Pr\big(\cL_{\frac{n}{4m}}=\emptyset\text{ and }N_{\frac{n}{4m}} \le M_{\frac{n}{4m}}\big) .\]
	Let $a$ be such that $\E \sigma_v+1>a>\E|B_{\tau_v}|$, and note that
	\[ \Pr\big(\cL_{\frac{n}{4m}}=\emptyset\text{ and }N_{\frac{n}{4m}} \le M_{\frac{n}{4m}}\big) \le \Pr\big(\cL_{\frac{n}{4m}}=\emptyset\text{ and }N_{\frac{n}{4m}} \le \tfrac{an}{4m}\big) + \Pr\big(M_{\frac{n}{4m}} \ge \tfrac{an}{4m}\big) .\]
	It remains to bound the terms on the right-hand side.
Note that, if $u \notin \cL$ then $L^\infty_u=\sigma_u(X)$. Thus, since $(\sigma_{v+i}(X)+1)_{i \in \Z}$ is an \iid\ sequence of bounded random variables with expectation strictly larger than $a$, standard large deviation bounds yield that $\Pr\big(\cL_{n/4m}=\emptyset\text{ and }N_{n/4m} \le \tfrac{an}{4m}\big)$ is exponentially small in $n$ (alternatively, we could appeal to Lemma~\ref{lem:stretched-exp-decay} with the sequence $(m-\sigma_{v+i}(X))_{i \in \Z}$ to obtain the required stretched-exponential bound).	
Towards establishing the bound on the second term, observe that, by~\eqref{eq:F_n} and~\eqref{eq:linear-stopping-process}, $(\tau_{v+i})_{i \in \Z}$ is a non-negative stopping-like stationary sequence with exponential tails. Thus, since $|B_n| = O(n^{d+1})$ by~\eqref{eq:linear-stopping-process}, Lemma~\ref{lem:stretched-exp-decay} implies that
\[ \Pr\big(M_{n/4m} \ge \tfrac{an}{4m}\big) \le e^{-\Omega\big(n^{1/(d+2)}\big)}\qquad\text{as }n \to \infty . \qedhere \]	
\end{proof}

\subsection{Proof of Lemma~\ref{lem:info-propagation}}
	Let $\cF_{v,r}$ denote the $\sigma$-algebra generated by $(X_{u,i})_{|u-v| \le r, 0 \le i \le \sigma_u}$. Set $r_0 := 0$ and let $r_n$ denote the smallest integer $r>r_{n-1}$ for which the random variables stated in the lemma are $\cF_{v,r}$-measurable. To prove the lemma, it suffices to show that $r_n \le r_{n-1} + 5\Delta n$ for $n \ge 1$, as this implies that $r_n \le 5\Delta n^2$. We henceforth abbreviate ``is $\cF$-measurable'' to ``is in $\cF$''.

	Let $n \ge 1$ and denote $r:=r_{n-1}$. We aim to show that $\mathbb{S}^n_v$ and $\{ Z^n_{w,j} \}_{|w-v| \le \Delta n, j \ge 0}$ are in $\cF_{v,r+5 \Delta n}$, using that $\mathbb{S}^{n-1}_v$ and $\{ Z^{n-1}_{w,j} \}_{|w-v| \le \Delta (n-1), j \ge 0}$ are in $\cF_{v,r}$. Throughout the proof, we repeatedly make use of the following easily verifiable properties:
	\[ v \preceq U^n_v \preceq v+ne_1 \qquad\text{and}\qquad (W^n_v,J^n_v) \in v+B_n \qquad\text{for all }v \in \Z^d\text{ and }n \ge 0 ,\]
	which, in particular, by~\eqref{eq:linear-stopping-process}, imply that
	\[ |U^n_v - v| \le n \qquad\text{and}\qquad |W^n_v-v| \le \Delta n \qquad\text{for all }v \in \Z^d\text{ and }n \ge 0 .\]
	
	\smallskip\noindent
	{\bf Step 1:} Consider step $n$ of the algorithm for $v$ and let $\cC_v^n \in \{\text{i},\dots,\text{iv}\}$ denote which case of the algorithm was executed. Let us show that $\cC_v^n$ is in $\cF_{v,r+n}$. To this end, we first check that the event $\{\cC_v^n = \text{i}\} = \{v\text{ is satified at time }n-1\} = \{(W^{n-1}_v,J^{n-1}_v) \notin B_{\tau_v(Z^{n-1})}\}$ is in $\cF_{v,r}$. Indeed, since $(W^{n-1}_v,J^{n-1}_v) \in v+B_{n-1}$, this event depends only on $(W^{n-1}_v,J^{n-1}_v)$ and $\{ Z^{n-1}_{w,j} \}_{(w,j) \in v+B_{n-1}}$, both of which are in $\cF_{v,r}$ by the definition of $r$ and by~\eqref{eq:linear-stopping-process}.
	
	Next, let us check that $L^{n-1}_{U^{n-1}_v}$ is in $\cF_{v,r+n}$. Since $U^{n-1}_v$ is in $\cF_{v,r}$ and since $v' \preceq U^{n-1}_{v'} \prec v'+ne_1$ for all $v'$, it suffices to check that the event $\{(v',u,i,w,j) \in D^{n-1}\text{ for some }w,j\}$ is in $\cF_{v,r+n}$ for any $v-ne_1 \preceq v' \preceq u \prec v+ne_1$ and any $i$.
	Indeed, this event is in $\cF_{v',r} \subset \cF_{v,r+|v'-v|} \subset \cF_{v,r+n}$.
	
	Finally, we check that the event that $U^{n-1}_v$ is loaded at time $n-1$ is in $\cF_{v,r+n}$.
	Note that, by~\eqref{eq:previous-source-locations-are-known}, $Y^{n-1}_{u,i}=X_{u,i}$ for $i \le L^{n-1}_u$ and $Y^{n-1}_{u,i}=\emptyset$ for $i>L^{n-1}_u$. Thus, since $U^{n-1}_v$ and $L^{n-1}_{U^{n-1}_v}$ are in $\cF_{v,r+n}$ and since $|U^{n-1}_v-v| \le n$, it follows that $(Y^{n-1}_{U^{n-1}_v,i})_{i \ge 0}$ is in $\cF_{v,r+n}$. Hence, as $\sigma$ is a simple stopping-process, we have that $\{ \sigma_{U^{n-1}_v}(Y^{n-1}) > L^{n-1}_{U^{n-1}_v} \}$ is in $\cF_{v,r+n}$, showing that the event that $U^{n-1}_v$ is loaded at time $n-1$ is in $\cF_{v,r+n}$.
	
	\smallskip\noindent
	{\bf Step 2:} Observe that $(U^n_v,I^n_v,W^n_v,J^n_v)$ is in $\cF_{v,r+n}$.
	Indeed, this follows from step 1, since in any case of the algorithm, $(U^n_v,I^n_v,W^n_v,J^n_v)$ is a deterministic function of $(U^{n-1}_v,I^{n-1}_v,W^{n-1}_v,J^{n-1}_v)$.
	
	\smallskip\noindent
	{\bf Step 3:}
	Let us check that $T^n_v$ is in $\cF_{v,r+3\Delta n}$. To this end, it suffices to check that the event that $v$ is the lexicographical-minimal element of $Q^{n-1}(W^{n-1}_v,J^{n-1}_v)$ is in $\cF_{v,r+3\Delta n}$, as $T^n_v$ equals 1 in this case and 0 otherwise. For this, it suffices to check that the set $Q^{n-1}(W^{n-1}_v,J^{n-1}_v)$ itself is in $\cF_{v,r+3\Delta n}$.
Since $(W^{n-1}_v,J^{n-1}_v)$ is in $\cF_{v,r}$ and since $|W^{n-1}_v-v| \le \Delta(n-1)$, it suffices to check that $Q^{n-1}(w,j)$ is in $\cF_{v,r+3\Delta n}$ for any $(w,j)$ such that $|w-v| \le \Delta n$. Fix such a $(w,j)$. We need to show that the event $\{v' \in Q^{n-1}(w,j)\}$ is in $\cF_{v',r+3\Delta n}$ for any $v' \in \Z^d$. Fix $v'$ and note that $v' \notin Q^{n-1}(w,j)$ unless $|v'-w| \le \Delta (n-1)$. Thus, we may assume that $|v'-w| \le \Delta (n-1)$. Then, by what we have shown in step~1 and since $Z^{n-1}_{w,j}$ is in $\cF_{v',r}$ (by the definition of $r$), the event $\{v' \in Q^{n-1}(w,j)\}$ is in $\cF_{v',r+n} \subset \cF_{w,r+n+|v'-w|}$.
	We have therefore shown that $Q^{n-1}(w,j)$ is in $\cF_{w,r+2\Delta n} \subset \cF_{v,r+2\Delta n+|w-v|} \subset \cF_{v,r+3\Delta n}$.

	\smallskip\noindent
	{\bf Step 4:} Observe that $\mathbb{S}^n_v$ is in $\cF_{v,r+3\Delta n}$. Indeed, since $D^n_v$ is determined by $\{ (U^t_v,I^t_v,W^t_v,J^t_v,T^t_v) \}_{1 \le t \le n}$, this follows immediately from steps 2 and 3.
	
	\smallskip\noindent
	{\bf Step 5:} Let us show that $\{ Z^n_{w,j} \}_{|w-v| \le \Delta n, j \ge 0}$ is in $\cF_{v,r+5\Delta n}$. To this end, let $(w,j)$ be such that $|w-v| \le \Delta n$ and denote
	\[ D^n_{w,j} := \{ (v',u,i) : (v',u,i,w,j) \in D^n \} .\]
	By step~4, for any $(v',u,i)$, the event $\{(v',u,i) \in D^n_{w,j} \}$ is in $\cF_{v',r+3\Delta n} \subset \cF_{w,r+3\Delta n+|v'-w|}$. Since $(v',u,i) \notin D^n_{w,j}$ unless $|v'-w| \le \Delta n$, we conclude that $D^n_{w,j}$ is in $\cF_{w,r+4\Delta n} \subset \cF_{v,r+4\Delta n+|w-v|} \subset \cF_{v,r+5\Delta n}$. Recall that $|D^n_{w,j}| \le 1$ by~\eqref{eq:once-written}. If $D^n_{w,j}=\emptyset$ then $Z^n_{w,j}=\emptyset$. Otherwise, $D^n_{w,j}=\{(v',u,i)\}$ and $Z^n_{w,j}=X_{u,i}$ for some $(v',u,i)$ such that $|u-v'| \le n$, in which case, $|u-v| \le |u-v'|+|v'-w|+|w-v| \le 3\Delta n$. It follows that $Z^n_{w,j}$ is in $\cF_{v,r+5\Delta n}$.
\qed

\section{Open questions}
\label{sec:open}

We have shown in Theorem~\ref{thm:ising} that the sub-critical Ising measure is \fvffiid\ with stretched-exponential tails, and we know from Remark~\ref{rem:ising-ffiid-exp-tails} that it is also \ffiid\ with exponential tails. The following question naturally arises:

\begin{question}
Let $d \ge 2$ and let $\mu$ be the unique Gibbs measure for the Ising model on $\Z^d$ at inverse temperature $\beta<\beta_c(d)$. Is $\mu$ \fvffiid\ with exponential tails?
\end{question}

A similar situation occurs in the more general setting of PCAs considered in Section~\ref{sec:pca}, where we have shown in Theorem~\ref{thm:main} that the limiting distribution of an exponentially uniformly ergodic PCA is \fvffiid\ with stretched-exponential tails, and we know from Theorem~\ref{thm:main0} that it is also \ffiid\ with exponential tails. A positive answer to the following natural question would yield a positive answer to the previous one:

\begin{question}
Let $\mu$ be the limiting distribution of an exponentially uniformly ergodic PCA (as defined in Section~\ref{sec:pca}). Is $\mu$ \fvffiid\ with exponential tails?
\end{question}

As we have mentioned in remarks after Theorem~\ref{thm:ising}, the critical Ising measure is known to be \ffiid, but is not known to be \fvffiid. This question was raised by van den Berg and Steif~\cite[Question~1]{van1999existence} and we reiterate it here:

\begin{question}
Let $d \ge 2$ and let $\mu$ be the unique Gibbs measure for the Ising model on $\Z^d$ at the critical inverse temperature $\beta=\beta_c(d)$. Is $\mu$ \fvffiid?
\end{question}

\bibliographystyle{amsplain}
%\nocite{*}
\bibliography{library}

\end{document}